\documentclass[12pt]{amsart}
\usepackage{amssymb}
\usepackage{amscd}
\usepackage{amsmath}
\usepackage{enumerate}
\usepackage[all]{xy}
\usepackage[pdftex]{graphicx}
\usepackage{verbatim}

\newtheorem{theo}{Theorem}
\newtheorem{lemma}[theo]{Lemma}
\newtheorem{cor}[theo]{Corollary}
\newtheorem{prop}[theo]{Proposition}

\newtheorem{defi}[theo]{Definition}

\newtheorem{conj}[theo]{Conjecture}
\newtheorem{remark}[theo]{Remark}

\newcommand{\bC}{{\mathbb{C}}}

\newcommand{\bR}{{\mathbb{R}}}

\newcommand{\calA}{{\mathcal{A}}}
\newcommand{\calB}{{\mathcal{B}}}

\newcommand{\calE}{{\mathcal{E}}}

\newcommand{\calU}{{\mathcal{U}}}
\newcommand{\calV}{{\mathcal{V}}}

\def\e{{\epsilon}}

\def\a{{\alpha}}
\def\b{{\beta}}

\newcommand{\B}{\mathbb{B}}
\newcommand{\BS}{\mathbb{S}}
\newcommand{\BD}{\mathbb{D}}

\begin{document}
\title{VANISHING POLYHEDRON AND COLLAPSING MAP}
\author{L\^e D\~ung Tr\'ang}
\author{Aur\'elio Menegon Neto}

\thanks{Partial support from CNPq (Brazil).}
\date{18-11-2015}
\address{L\^e D\~ung Tr\'ang: Universit\'e Aix-Marseille - France.}
\email{ledt@ictp.it}
\address{Aur\'elio Menegon Neto: Universidade Federal da Para\'iba - Brazil.}
\email{aurelio@mat.ufpb.br}

\begin{abstract}
In this paper we give a detailed proof that the Milnor fiber $X_t$ of an analytic complex isolated singularity function defined on a reduced $n$-equidimensional analytic complex space $X$ is a regular neighborhood of a polyhedron $P_t \subset X_t$ of real dimension $n-1$. Moreover, we describe the degeneration of $X_t$ onto the special fiber $X_0$, by giving a continuous collapsing map $\Psi_t: X_t \to X_0$ which sends $P_t$ to $\{0\}$ and which restricts to a homeomorphism $X_t \backslash P_t \to X_0 \backslash \{0\}$.
\end{abstract}

\maketitle

\section*{Introduction}

Let $f: (X,x) \to (\bC,0)$ be a germ of complex analytic function $f$ at a point $x$ of a reduced equidimensional complex analytic space $X \subset \bC^N$ (with arbitrary singularity). In \cite{Le3} the first author proved that there exist sufficiently small positive real numbers $\e$ and $\eta$ with $0<\eta \ll \e \ll 1$ such that the restriction:
$$f_|: \B_\e(x) \cap X \cap f^{-1}(\BD_\eta^*) \to \BD_\eta^*$$
is a locally trivial topological fibration, where $\B_\e(x)$ is the ball around $x \in \bC^N$ with radius $\e$, $\BD_\eta$ is the disk around $0 \in \bC$ with radius $\eta$ and $\BD_\eta^* := \BD_\eta \backslash \{0\}$.

\medskip
For any $t \in \BD_\eta^*$ the topology of $X_t := \B_\e(x) \cap X \cap f^{-1}(t)$ does not depend on $\e$ small enough (see Theorem 2.3.1 of \cite{LT}). We call $X_t$ the Milnor fiber of $f$, with boundary $\partial X_t := X_t \cap \BS_\e(x)$. Also set $X_0 := \B_\e(x) \cap X \cap f^{-1}(0)$.

\medskip
Let ${\mathcal S} = (S_\a)_{\a \in A}$ be a Whitney stratification of $X$ (\cite{Wh}, \cite{Ma}). According to Definition 1.1 of \cite{Le5}, we say that $f: (X,x) \to (\bC,0)$ has an isolated singularity at $x \in X$ in the stratified sense, i.e., relatively to the Whitney stratification ${\mathcal S}$, if there is an open neighborhood $U$ of $x$ in $X$ such that the restriction of $f$ to each stratum $S_\a \cap U$ that does not contain $x$ is a submersion and if the restriction of $f$ to the stratum $S_{\a(x)} \cap U$ that contains $x$ has an isolated singularity at $x$.

\medskip
The first author sketched a proof of the following theorem in \cite{Le4}:

\begin{theo} \label{theo_main}
Let $X \subset \bC^N$ be a reduced equidimensional complex analytic space and let ${\mathcal S} = (S_\a)_{\a \in A}$ be a Whitney stratification of $X$. Let $f: (X,x) \to (\bC,0)$ be a germ of complex analytic function $f$ at a point $x \in X$. If $f$ has an isolated singularity at $x$ relatively to ${\mathcal S}$ and if $\e$ and $\eta$ are sufficiently small positive real numbers as above, then for each $t \in \BD_\eta^*$ there exist:
\begin{itemize}
\item[$(i)$] a polyhedron $P_t$ of real dimension $\dim_\bC X_t$ in the Milnor fiber $X_t$, compatible with the Whitney stratification $\mathcal{S}$, and a continuous simplicial map: 
$$\tilde{\xi}_t: \partial X_t \to P_t$$
compatible with $\mathcal{S}$, such that $X_t$ is homeomorphic to the mapping cylinder of $\tilde{\xi}_t$;

\medskip
\item[$(ii)$] a continuous map $\Psi_t: X_t \to X_0$ that sends $P_t$ to $\{0\}$ and that restricts to a homeomorphism $X_t \backslash P_t \to X_0 \backslash \{0\}$.
\end{itemize}
\end{theo}

The purpose of this paper is to give a complete and detailed proof of Theorem \ref{theo_main}. This theorem was conjectured by R. Thom in the early 70's when $X$ is smooth.

\medskip

In section \ref{section_polar_curves} we construct the {\it relative polar curve} of $f$, which is the main tool to prove Theorem \ref{theo_main}. In section \ref{section_two_dimensional_case} we prove the Theorem when $X$ is two-dimensional. Then in section \ref{section_main_theorem} we prove the Theorem in the general case, using two propositions (Proposition \ref{prop_a} and Proposition \ref{prop_b}). In section \ref{section_prop_a} we prove those Propositions by finite induction on the dimension of $X$. Finally, in section \ref{section_prop_Et} we make the detailed construction of a vector field (Lemma \ref{prop_Et}) that is used in section \ref{section_prop_a}.

\medskip
\section{Polar curves}
\label{section_polar_curves}

The notion of polar curve of a complex analytic function defined on an open neighborhood of $\bC^n$ relatively to a linear form $\ell$ were introduced by B. Teissier and the first author in \cite{Le1} and \cite{Te}. Later, in \cite{Le2} M. Kato and the first author extended this concept to a germ of complex analytic function $f: (X,x) \to (\bC,0)$ relatively to a Whitney stratification ${\mathcal S} = (S_\a)_{\a \in A}$ of a reduced equidimensional complex analytic space $X$. 

\medskip
Notice that by now we are not supposing that $f$ has an isolated singularity in the stratified sense. This hypothesis will be asked after the lemma below.

\medskip
Let $f: X \to \bC$ be a representative of the germ of function at $x$ such that $X$ is closed in an open neighborhood $U$ of $x$ in $\bC^N$. For any linear form:
$$\ell: \bC^N \to \bC $$
the function $f$ and the restriction of $\ell$ to $X$ induce the analytic morphism:
$$\phi_\ell: X \to \bC^2$$
defined by $\phi_\ell(z) = \big( \ell(z), f(z) \big)$, for any $z \in X$. 

\medskip
We have the following lemma:

\begin{lemma} \label{lemma_polar_curve}
There is a representative $X$ of $(X,x)$ in an open neighborhood $U$ of $\bC^n$ and a non-empty Zariski open set $\Omega$ in the space of non-zero linear forms of $\bC^N$ to $\bC$ that take $x \in \bC^N$ to $0 \in \bC$ such that, for any $\ell \in \Omega$, the analytic morphism $\phi_\ell: X \to \bC^2$ satisfies: 
\begin{itemize}
\item[$(i)$] The part of the critical locus of the restriction of $\phi_\ell$ to a stratum $S_\a$ that lies outside $f^{-1}(0)$ is either empty or a smooth reduced complex curve, whose closure in $X$ is denoted by $\Gamma_\a$.
\item[$(ii)$] The image $(\Delta_\a,0)$ of $(\Gamma_\a,x)$ by $\phi_\ell$ is the germ of a complex curve.
\end{itemize}
\end{lemma}

\begin{proof}
Let us choose an open neighborhood $U$ in $\bC^n$ such that the intersection $U \cap S_\a$ is not empty for finite many indices $\a$. Furthermore, we may assume that the closure $\overline{S}_\a$ in $U$ is defined by an ideal $I(\overline{S}_\a)$ generated by complex analytic functions $g_1, \dots, g_m$ defined on $U$, that is, $I(\overline{S}_\a) = (g_1, \dots, g_m)$.

\medskip
Now consider a linear form $\ell = a_1 x_1 + \dots + a_N x_N$ and a stratum $S_\a$ such that $0 \in \overline{S}_\a$. Let $C_{\ell,\a}$ be the critical set of the restriction of $\phi_\ell$ to $S_\a \setminus f^{-1}(0)$. We will show that $C_{\ell,\a}$ is contained in a member of a linear system parametrized by $\ell$. Consider the matrix:

$$J_\a = 
\begin{pmatrix}
\partial g_1 / \partial x_1 & \dots & \partial g_1 / \partial x_N \\
\vdots & \ddots & \vdots \\
\partial g_m / \partial x_1 & \dots & \partial g_m / \partial x_N \\
\end{pmatrix}
.$$ 

A point $z$ of $S_\a$ is a point where the rank of $J_\a$ at $z$ is $\rho := \displaystyle\max_{z \in \overline{S}_\a} \ \text{rank} \big( J_\a (z) \big)$, since it is a non-singular point of $\overline{S}_\a$. A point of $C_{\ell,\a}$ is a point of $S_\a \setminus f^{-1}(0)$ where the matrix:

$$J_{\phi,\a} = 
\begin{pmatrix}
\partial g_1 / \partial x_1 & \dots & \partial g_1 / \partial x_N \\
\vdots & \ddots & \vdots \\
\partial g_m / \partial x_1 & \dots & \partial g_m / \partial x_N \\
\partial f / \partial x_1 & \dots & \partial f / \partial x_N \\
a_1 & \dots & a_N \\
\end{pmatrix}
$$ 
has rank at most $\rho+1$. So the determinants of the $(\rho+2)$-minors:

$$
\begin{pmatrix}
\partial g_{i_1} / \partial x_{j_1} & \dots & \partial g_{i_1} / \partial x_{j_{\rho+2}} \\
\vdots & \ddots & \vdots \\
\partial g_{i_\rho} / \partial x_{j_1} & \dots & \partial g_{i_\rho} / \partial x_{j_{\rho+2}} \\
\partial f / \partial x_{j_1} & \dots & \partial f / \partial x_{i_{\rho+2}} \\
a_{i_1} & \dots & a_{i_{\rho+2}} \\
\end{pmatrix}
$$
are zero, that is:
$$
\sum_{k=1}^{\rho+2} (-1)^{k+1} \cdot a_{i_k} \cdot \text{det}
\begin{pmatrix}
\partial g_{i_1} / \partial x_{j_1} & \dots & \partial g_{i_1} / \partial x_{j_{k-1}} & \partial g_{i_1} / \partial x_{j_{k+1}} & \dots & \partial g_{i_1} / \partial x_{j_{\rho+2}} \\
\vdots & \ddots & \vdots & \vdots & \ddots & \vdots \\
\partial g_{i_\rho} / \partial x_{j_1} & \dots & \partial g_{i_\rho} / \partial x_{j_{k-1}} & \partial g_{i_\rho} / \partial x_{j_{k+1}} & \dots & \partial g_{i_\rho} / \partial x_{j_{\rho+2}} \\
\partial f / \partial x_{j_1} & \dots & \partial f / \partial x_{j_{k-1}} & \partial f / \partial x_{j_{k+1}} & \dots & \partial f / \partial x_{j_{\rho+2}} \\
\end{pmatrix}
=0.$$

\medskip
An analytic version of a classical theorem of Bertini (see \cite{Be1} and \cite{Be2}) states that if $h_1, \dots, h_r$ are holomorphic functions defined on a complex analytic space $Y$ and if $\lambda_i$ are sufficiently generic, for $i=1, \dots, r$, then the singular locus of the subvariety $\left\{ \sum_{i=1}^r \lambda_i h_i =0 \right\}$ is contained in the union of the singular set of $Y$ and the set:
$$\{h_1 = \dots = h_r =0 \} \, .$$

\medskip
So it follows from the analytic theorem of Bertini that there exists a non-empty Zariski open set $\Omega_\a$ in the space of non-zero linear forms of $\bC^N$ to $\bC$ that take $x \in \bC^N$ to $0 \in \bC$ such that for any $\ell \in \Omega_\a$ one has that the singular points $\Sigma_{C_{\ell,\a}}$ of $C_{\ell,\a}$ are contained in the union of the set of the points where the determinants above are zero and the singular locus of $S_\a$. That is:
$$\Sigma_{C_{\ell,\a}} \subset \big( \text{Crit} \ (f_{| \overline{S}_\a}) \cup \Sigma_{\overline{S}_\a} \big) \cap S_\a \, .$$
Since this intersection is contained in $f^{-1}(0) \cap S_\a$ and $C_{\ell,\a}$ is contained in $S_\a \setminus f^{-1}(0)$, it follows that $C_{\ell,\a}$ is smooth. 

\medskip
Now, since $\phi_\ell^{-1}(0,0) \cap (\Gamma_\a,x) \subset \{x\}$, it follows from the geometric version of the Weierstrass preparation theorem given in \cite{Ho} that the restriction of $\phi_\ell$ to the germ $(\Gamma_\a,x)$ is finite. In particular, a theorem of Remmert (see \cite{Gr}, page 60, for instance) implies that the image $\Delta_\a$ of the analytic set $\Gamma_\a$ by the finite morphism $\phi_\ell$ is an analytic set of the same dimension, and hence it is a complex curve. 

\medskip
Finally, notice that there is a finite number of indices $\a \in A$ such that $x$ is contained in $\overline{S}_\a$. Let $A_x$ be the finite subset of $A$ formed by such indices. So the set:
$$\Omega := \bigcap_{\a \in A_x} \Omega_\a$$
is the desired non-empty Zariski open set in the space of non-zero linear forms of $\bC^N$ to $\bC$ that take $x \in \bC^N$ to $0 \in \bC$.

\end{proof}

For any $\ell \in \Omega$ we say that the germ of curve at $x$ given by:
$$\Gamma_\ell := \bigcup_{\a \in A} \Gamma_\a$$
is the {\it polar curve of $f$ relatively to $\ell$ at $x$} and that the germ of curve at $0$ given by: 
$$\Delta_\ell := \bigcup_{\a \in A} \Delta_\a$$
is the {\it polar discriminant of $f$ relatively to $\ell$ at $0$}. In fact, $\Delta_\ell$ behaves as a discriminant of the germ $\phi_\ell$ as indicated in Proposition \ref{prop_discriminant} below.

\medskip
From now on, we fix a linear form $\ell \in \Omega$ and we set $\phi := \phi_\ell$, $\Gamma := \Gamma_\ell$ and $\Delta := \Delta_\ell$. 

\medskip
Also, from this point we will assume that $f$ has an isolated singularity at $x$ relatively to the stratification $\mathcal{S}$.

\medskip
Now let us recall some definitions and a result that gives us an important tool to prove Theorem \ref{theo_main} (see \cite{LT} for instance). 

\begin{defi} \label{defi_svf}
Let $\mathcal{X}$ and $\mathcal{Y}$ be analytic spaces with Whitney stratifications $(\mathcal{X}_\a)_{\a \in A}$ and $(\mathcal{Y}_\b)_{\b \in B}$ respectively. A map $h: \mathcal{X} \to \mathcal{Y}$ is a stratified map if:
\begin{itemize}
\item[$(i)$] $h$ is continuous;
\item[$(ii)$] $h$ sends each stratum $\mathcal{X}_\a$ to a unique stratum $\mathcal{Y}_{\b(\a)}$, for some $\b(\a) \in B$;
\item[$(iii)$] the restriction of $h$ to each stratum $\mathcal{X}_\a$ induces a smooth map $h_\a: \mathcal{X}_\a \to \mathcal{Y}_{\b(\a)}$.
\end{itemize}

We say that a stratified map $h$ as above is a stratified submersion if each $h_\a$ is a submersion.

We say that a stratified map $h$ as above is a stratified homeomorphism if $h$ is a homeomorphism and each $h_\a$ is a smooth diffeomorphism.
\end{defi}

\begin{defi}
Let $\mathcal{X}$ be an analytic space with Whitney stratification $(\mathcal{X}_\a)_{\a \in A}$. We say that a continuous vector field $v$ in $\mathcal{X}$ is a stratified vector field if for each $\a \in A$, the restriction $v_\a$ of $v$ to $\mathcal{X}_\a$ is a smooth vector field. 
\end{defi}

We have:

\begin{lemma} \label{lemma_Ve}
Let $\mathcal{X}$ be a complex analytic space with a Whitney stratification $(\mathcal{X}_\a)_{\a \in A}$, and let $h: \mathcal{X} \to U$ be a proper stratified submersion over an open subset $U$ of $\bC^k$. If $v$ is a smooth vector field in $U$, then $h$ lifts $v$ to a stratified vector field in $\mathcal{X}$ that is integrable.
\end{lemma}

The proof comes from the fact that, under the hypothesis of the first isotopy theorem of Thom-Mather (\cite{Ma}), a smooth vector field lifts to a rugose vector field, and hence integrable. For this, we notice that for a complex analytic space, Theorem 1.2 of chapter V of \cite{Te2} gives that the Whitney conditions imply the strict Whitney conditions, and then our claim follows from Proposition 4.6 of \cite{Ve}.

\medskip
As we said before, Lemma \ref{lemma_Ve} will be used many times in the next sections of this paper.

\medskip
Now let us go back to the map $\phi = (\ell, f): X \to \bC^2$ defined above. Notice that it is stratified and that it induces a stratified submersion $\phi^{-1}(U \setminus \Delta) \to U \setminus \Delta$, where $U$ is an open subset of $\bC^2$.

\medskip
By an analytic version of Corollary 2.8 of \cite{Mi}, there exist $\e$ and $\eta_2$ small enough positive real numbers with $0< \eta_2 \ll \e \ll 1$ such that, for any $t \in \BD_{\eta_2}$, the sphere $\BS_\e(x)$ around $x$ of radius $\e$ intersects $f^{-1}(t) \cap S_\a$ transversally, for any $\a \in A$.

\medskip
We can also choose the linear form $\ell$ in such a way that there exists $\eta_1$ sufficiently small, with $0< \eta_2 \ll \eta_1 \ll \e \ll 1$, such that $\phi^{-1}(s,t) \cap S_\a = \ell^{-1}(s) \cap f^{-1}(t) \cap S_\a$ intersect $\BS_\e(x)$ transversally, for any $(s,t) \in \BD_{\eta_1} \times \BD_{\eta_2}$, where $\BD_{\eta_1}$ and $\BD_{\eta_2}$ are the closed disks in $\bC$ centered at $0$ and with radii $\eta_1$ and $\eta_2$, respectively.

\medskip
So we have:

\begin{prop} \label{prop_discriminant}
The map $\phi=(\ell,f)$ induces a stratified submersion: 
$$\phi_|: \B_\e(x) \cap X \cap \phi^{-1}(\BD_{\eta_1} \times \BD_{\eta_2} \setminus \Delta) \to \BD_{\eta_1} \times \BD_{\eta_2} \setminus \Delta \, .$$
\end{prop}

In particular, the first isotopy theorem of Thom-Mather gives:

\begin{cor} \label{cor_discriminant}
The restriction $\phi_|$ above is a topological locally trivial fibration.
\end{cor}

Therefore the curve $\Delta$ plays the role of a local topological discriminant for the stratified map $\phi$.
  
\medskip
For any $t$ in the disk $\BD_{\eta_2}$ set: 
$$D_t := \BD_{\eta_1} \times \{t\} \, .$$
If $t \neq 0$, the Milnor fiber $f^{-1}(t) \cap \B_\e(x)$ of $f$ is homeomorphic to $\phi^{-1}(D_t) \cap \B_\e(x)$ (see Theorem 2.3.1 of \cite{LT}). So, in order to simplify notation, we reset: 
$$X_t := \phi^{-1}(D_t) \cap \B_\e(x) \, .$$
Notice that with this notation, the boundary $\partial X_t$ of $X_t$ is given by the union of $\phi^{-1}(\mathring{D}_t) \cap \BS_\e(x)$ and $\phi^{-1}(\partial D_t) \cap \B_\e(x)$.

\medskip
Next we give a lemma that will be implicitly used many times in the paper:

\begin{lemma} \label{lemma_strata}
Let $\mathcal{X}$ and $\mathcal{Y}$ be two analytic spaces in $\bC^N$ such that $\mathcal{X}$ has a Whitney stratification $(\mathcal{X}_\a)_{\a \in A}$ and such that $\mathcal{Y}$ is non-singular. If $\mathcal{Y}$ intersects each stratum $\mathcal{X}_\a$ transversally in $\bC^N$, then the Whitney stratification of $\mathcal{X}$ induces a Whitney stratification $(\mathcal{Y}_\a)_{\a \in A}$ of $\mathcal{X} \cap \mathcal{Y}$, where $\mathcal{Y}_\a := \mathcal{X}_\a \cap \mathcal{Y}$, for each $\a \in A$.
\end{lemma}

In particular, the Whitney stratification ${\mathcal S}$ of $X$ induces a Whitney stratification $\mathcal{S}(t)$ of $X_t$. Precisely, the strata of such Whitney stratification are the following intersections, for $\a \in A$:
\begin{itemize}
\item[$(i)$] $\mathcal{S}_\a \cap (X_t \setminus \partial X_t)$;
\item[$(ii)$] $\mathcal{S}_\a \cap \phi^{-1}(\mathring{D}_t) \cap \BS_\e(x)$;
\item[$(iii)$] $\mathcal{S}_\a \cap \phi^{-1}(\partial D_t) \cap \mathring{\B}_\e(x)$;
\item[$(iv)$] $\mathcal{S}_\a \cap \phi^{-1}(\partial D_t) \cap \BS_\e(x)$.
\end{itemize}

\medskip
Now, for any $t \in \BD_{\eta_2}^*$ set $D_t := \BD_{\eta_1} \times \{t\}$. Then $\phi$ induces a stratified map:
$$\ell_t: X_t \to D_t \, .$$
By construction, the restriction of $\ell_t$ to each stratum of $X_t$ is a submersion onto the image of the stratum at any point out of $\Gamma$. Therefore it induces a locally trivial fibration over $D_t \backslash (\Delta \cap D_t)$. That is, if we set:
$$\Delta \cap D_t  = \{y_1(t), \dots, y_k(t) \}$$
then the restriction of $\ell_t$ given by:
$$\varphi_t: X_t \setminus \ell_t^{-1} \big( \{y_1(t), \dots, y_k(t) \} \big) \to D_t \setminus \{y_1(t), \dots, y_k(t) \}$$
is a stratified submersion (see Definition \ref{defi_svf}) and a locally trivial fibration, by Thom-Mather first isotopy theorem.

\begin{remark}
In the case that $\Gamma$ is empty, one has that: 
$$\phi_|: \B_\e(x) \cap X \cap \phi^{-1} (\BD_{\eta_1} \times \BD_{\eta_2}) \to \BD_{\eta_1} \times \BD_{\eta_2} $$
is a locally trivial topological fibration, which implies a locally trivial topological fibration $\ell_t: X_t \to D_t$. Hence in this case the Milnor fiber $X_t$ is homeomorphic to the product of $D_t$ and the general fiber of $\ell_t$. 
\end{remark}

So from now on we shall assume that the polar curve $\Gamma$ is not empty.

\medskip
\section{The two-dimensional case}
\label{section_two_dimensional_case}

We shall prove Theorem \ref{theo_main} by induction on the dimension $n$ of the analytic space $X$. We could start by proving the theorem for $n=1$ and then proceed by induction for $n\geq 2$, but we choose to start with the $2$-dimensional case, in order to provide the reader a better intuition of the constructions.

\medskip
So in this section we prove our theorem when $(X,0)$ is a $2$-dimensional reduced equidimensional germ of complex analytic space and $f: (X,0) \to (\bC,0)$ has an isolated singularity at $0$ in the stratified sense. 

\medskip
One particularity of this $2$-dimensional case is that the singular set $\Sigma$ of $X$ has dimension at most one. If $\Sigma$ has dimension one, we can put it inside the curve $\Gamma$. Precisely, (only) in this section we denote by $\Gamma$ the union of the polar curve of $f$ with $\Sigma$. We also denote by $\Delta$ the union of the polar discriminant of $f$ with $\phi(\Sigma)$. Notice that if $\Delta$ is not empty, it is a complex curve.

\vspace{0.3cm}
\subsection{First step: constructing the vanishing polyhedron}
\ \\

For any $t \in \BD_{\eta_2}^*$ fixed, we remind that:
$$\Delta \cap D_t  = \{y_1(t), \dots, y_k(t) \} \, .$$
Let $\lambda_t$ be a point in $D_t \setminus \{y_1(t), \dots, y_k(t)\}$ and for each $j = 1, \dots, k$, let $\delta(y_j(t))$ be the line segment in $D_t$ starting at $\lambda_t$ and ending at $y_j(t)$. We can choose $\lambda_t$ in such a way that any two of these line segments intersect only at $\lambda_t$.

\medskip
Set: 
$$Q_t := \bigcup_{j=1}^k \delta(y_j(t))$$
and
$$P_t := \ell_t^{-1}(Q_t) \, .$$


\medskip
Since $\ell_t$ is finite, one can see that $P_t$ is a one-dimensional polyhedron in $X_t$. And since the map $\varphi_t: X_t \setminus \ell_t^{-1} \big( \{y_1(t), \dots, y_k(t) \} \big) \to D_t \setminus \{y_1(t), \dots, y_k(t) \}$ is a stratified submersion, each $1$-simplex of $P_t$ is contained in some stratum $X_t \cap \mathcal{S}_\a$ of $X_t$, so $P_t$ is adapted to the stratification $\mathcal{S}$. 

\medskip
We say that $P_t$ is a {\it vanishing polyhedron for $f$}.

\medskip
Notice that, in this $2$-dimensional case, the vanishing polyhedron $P_t$ contains the $0$-dimensional strata of the Whitney stratification of $X_t$ induced by $\mathcal{S}$, which are the points of $\Sigma \cap X_t$. In particular, $X_t \backslash P_t$ is a smooth manifold.

\begin{lemma} \label{lemma_v_t}
There exists a continuous vector field $v_t$ in $D_t$ such that:
\begin{enumerate}
\item It is smooth on $D_t \setminus Q_t$;
\item It is null on $Q_t$;
\item It is  transversal to $\partial D_t$ and points inwards;
\item The associated flow $q_t: [0, \infty) \ \times \ (D_t \backslash Q_t) \to D_t \backslash Q_t$ defines a map:
$$
\begin{array}{cccc}
\xi_t \ : & \! \partial D_t & \! \longrightarrow & \! Q_t \\
& \! u & \! \longmapsto & \! \displaystyle \lim_{\tau \to \infty} q_t(\tau,u)
\end{array}
,$$
such that $\xi_t$ is continuous, simplicial and surjective.
\end{enumerate}
\end{lemma}

\begin{proof}
Let $d_t: D_t \to \bR$ be the function given by the distance to the set $Q_t$, that is $d_t(x) := d(x, Q_t)$. Consider the small closed neighborhood of $Q_t$ in $D_t$ given by: 
$$\mathcal{R}_t:= (d_t)^{-1}\big( [0,r] \big) $$
for some small $r>0$. 

\medskip
Notice that the boundary of $\mathcal{R}_t$ is a Jordan curve. In fact, let $d_t^j$ be the function given by the distance to $\delta(y_j(t))$, which is smooth outside $\delta(y_j(t))$ since $\delta(y_j(t))$ is a smooth manifold (see \cite{Fo} for instance). So $(d_t^j)^{-1}\big( (0,r] \big)$ is a smooth manifold $\mathcal{R}_t(j)$ with boundary diffeomorphic to the circle. Clearly, $\mathcal{R}_t = \cup_{j=1}^k \mathcal{R}_t(j)$ and its boundary is a piecewise linear closed curve.

\medskip
We endow $\mathcal{R}_t$ with the vector field $v_1$ given as follows:

\medskip
Let us decompose $\mathcal{R}_t$ in the sets $R_t^1, \dots, R_t^k, M_t^1, \dots, M_t^k$ as indicated in Figure \ref{Fig2} below.

\begin{figure}[!h]
\centering
\includegraphics[scale=1]{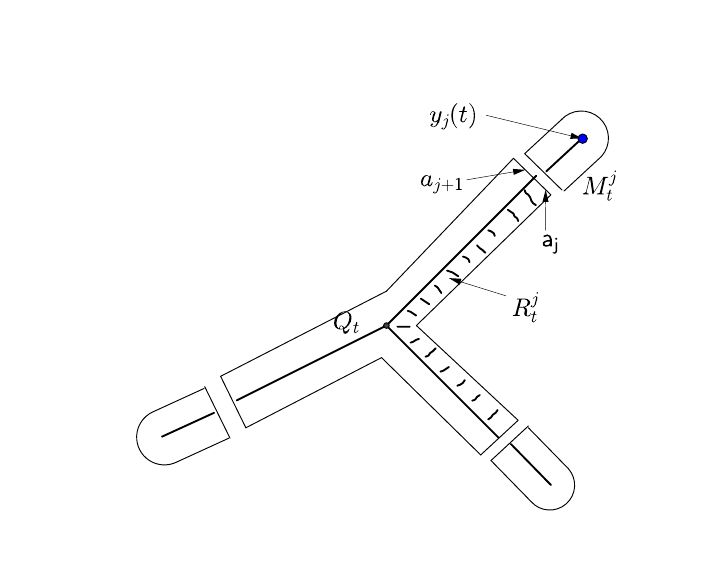}
\caption{}
\label{Fig2}
\end{figure}

We endow each ``rectangular'' region $R_t^j$ with an integrable vector field $\omega_j$ as follows: 

\medskip
Let $\mathbb{I}$ denote the closed interval $[0,1]$. Using the Riemann mapping and Carath\'eodory's theorem (see sections 17.3 and 17.20 of \cite{NP}, for instance, or see \cite{Ca} for the original proof), we can consider a homeomorphism $\tilde{h}_j: R_t^j \to \mathbb{I} \times \mathbb{I}$ such that: 
\begin{itemize}
\item[$\circ$] $\tilde{h}_j$ takes the side $Q_t \cap R_t^j$ onto $\mathbb{I} \times \{0\}$;
\item[$\circ$] $\tilde{h}_j$ takes the side $a_j$ onto $\{0\} \times \mathbb{I}$;
\item[$\circ$] $\tilde{h}_j$ restricts to a diffeomorphism $h_j: int(R_t^j) \to (0,1) \times (0,1)$, where $int(R_t^j)$ denotes the interior of $R_t^j$.
\end{itemize}

Also let $\rho: \mathbb{I} \times \mathbb{I} \to \bR$ be the function given by the distance to $\mathbb{I} \times \{0\}$, that is, $\rho(u,v) := v$. Since $\rho$ is smooth, the vector field $\omega$ on $\mathbb{I} \times \mathbb{I}$ given by the opposite of the gradient vector field associated to $\rho$ is continuous, smooth, non-zero outside $\mathbb{I} \times \{0\}$, zero on $\mathbb{I} \times \{0\}$, and each orbit associated to it has a limit point on $\mathbb{I} \times \{0\}$.

\medskip
So the pull-back of $\omega$ by $h_j$ gives a vector field $\omega_j$ on $R_t^j$ that is continuous, smooth and non-zero outside $Q_t \cap R_t^j$, zero on $Q_t \cap R_t^j$, and each orbit associated to it has a limit point on $Q_t \cap R_t^j$. 

\medskip
Since $h_j^{-1}$ sends each integral curve associated to $\omega$ to an integral curve associated to $w_j$ and since $h_j^{-1}$ sends the side $\{0\} \times \mathbb{I}$ to the side $a_j$, then the fact that $\{0\} \times \mathbb{I}$ coincides with an integral curve implies that the line segment $a_j$ coincides with an integral curve associated to $w_j$.

\medskip
Putting all the rectangles $R_t^j$ together, each of them endowed with the corresponding vector field $\omega_j$, we get a subset $R_t$ of $\mathcal{R}_t$ endowed with a continuous vector field $w$ that is smooth and non-zero outside $Q_t$, zero on $Q_t$, and each orbit associated to it has a limit point on $Q_t$. Moreover, each line segment $\tilde{a}_j = a_j \cup a_{j+1}$ coincides with an integral curve associated to $w$ (for consistence of the notation, we set $a_{k+1}:= a_1$).

\medskip
Now we do a similar process to endow each component $M_t^j$ with a vector field $w_j$ that is continuous, smooth and non-zero outside $Q_t \cap M_t^j$, zero on $Q_t \cap M_t^j$, and each orbit associated to it has a limit point on $Q_t \cap M_t^j$. Even more, we will do it in such a way that $w_j$ ``glues smoothly" with the vector field $w$ of $R_t$.

\medskip
Consider the set $A \subset \bR^2$ given by the union of the rectangle $\mathbb{I} \times \mathbb{I}$ and a semi-disk as in Figure \ref{FigA}. Also consider a diffeomorphism $g_j: M_t^j \to A$ that takes the side $\tilde{a}_j$ onto $\{0\} \times \mathbb{I}$ and that takes $Q_t \cap M_t^j$ onto $\mathbb{I} \times \{1/2\}$.

\medskip
Now consider the subsets $A_1 := [0,2/3] \times [0,1]$ and $A_2 := A \setminus \big( [0,1/3) \times [0,1] \big)$ of $A$. 

\medskip
Recall that the side $\tilde{a}_j$ of $M_t^j$ is endowed with a vector field given by the restriction of the vector field $\omega$ to $\tilde{a}_j \subset R_t$. This naturally gives a vector field $\upsilon_1$ in $A_1$ that is continuous, smooth and non-zero outside $[0,2/3] \times \{1/2\}$, zero on $[0,2/3] \times \{1/2\}$, and each integral curve associated to it is a line segments $\{u\} \times \mathbb{I}$, for some $u \in [0,2/3]$, which has a limit point $(u,0)$. 

\medskip
On the other hand, we endow $A_2$ with the vector field $\upsilon_2$ given by the opposite of the gradient vector field of the function distance to the line segment $[1/3,1] \times \{1/2\}$. It is continuous, smooth and non-zero outside $[1/3,1] \times \{1/2\}$, zero on $[1/3,1] \times \{1/2\}$, and each integral curve associated to it has a limit point in $[1/3,1] \times \{1/2\}$. 

\medskip
Then we glue the vector fields $\upsilon_1$ and $\upsilon_2$, using a partition of the unity, to obtain a vector field $\upsilon$ on $A$ that is continuous, smooth and non-zero outside $\mathbb{I} \times \{1/2\}$, zero on $\mathbb{I} \times \{1/2\}$, and each orbit associated to it has a limit point on $\mathbb{I} \times \{1/2\}$. 

\medskip
Then we define the vector field $w_j$ on $M_t^j$ as the pull-back of $\upsilon$ by $g_j$.

\begin{figure}[!h]
\centering
\includegraphics[scale=0.15]{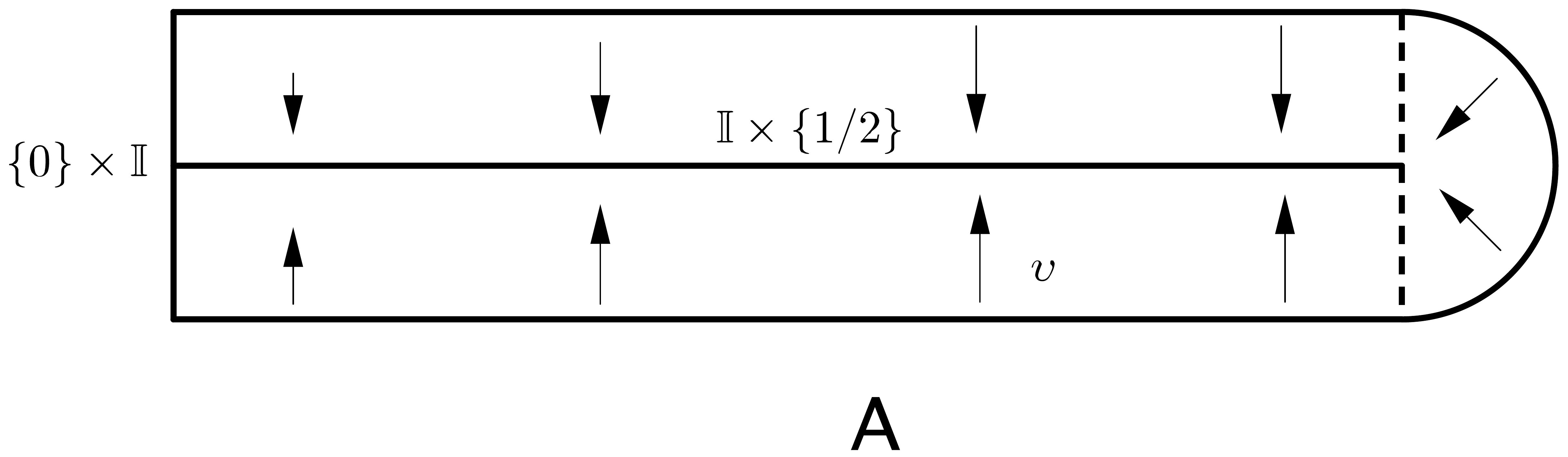}
\caption{}
\label{FigA}
\end{figure}

\medskip
So gluing the vector field $w$ and the vector fields $w_j$ we obtain a vector field $v_1$ on $\mathcal{R}_t$ that is continuous, smooth and non-zero outside $Q_t$, zero on $Q_t$, and each orbit associated to it has a limit point on $Q_t$.

\medskip
On the other hand, let $r'$ be a small real number with $0<r'<r$ and with $r-r' \ll 1$, and set $\mathcal{R}_t':= (d_t)^{-1}\big( [0,r'] \big) $, so $\mathcal{R}_t' \subset \mathcal{R}_t$. Also denote by $int(\mathcal{R}_t)$ the interior of $\mathcal{R}_t$ and by $int(\mathcal{R}_t')$ the interior of $\mathcal{R}_t'$. We endow $D_t \backslash int(\mathcal{R}_t')$ with the vector field $v_2$ given by the opposite of the gradient vector field of the function on $D_t \backslash int(\mathcal{R}_t')$ given by the distance to the point $\lambda_t$.

\medskip
Finally, one can check that the vector fields $v_1$ and $v_2$ never have opposite directions. So the vector field $v_t$ is obtained by gluing the vector fields $v_1$ and $v_2$, using a partition of unity. That is, we consider a pair $(\rho_1, \rho_2)$ of continuous functions from the compact disk $D_t$ to the closed unit interval $[0,1]$ such that:
\begin{itemize}
\item[$\circ$] for every point $p \in D_t$ one has that $\rho_1(p) + \rho_2(p) = 1$ ;
\item[$\circ$] the support of $\rho_1$ is contained in $D_t \backslash \mathcal{R}_t'$ ;
\item[$\circ$] the support of $\rho_2$ is contained in $int(\mathcal{R}_t)$ .
\end{itemize}
Hence for any $p \in D_t \backslash int(\mathcal{R}_t)$ we have that $\big( \rho_1(p), \rho_2(p) \big) = (1,0)$ and for any $p \in \mathcal{R}_t'$ we have that $\big( \rho_1(p), \rho_2(p) \big) = (0,1)$. So we set $v_t := \rho_1 v_1 + \rho_2 v_2$ . 

\medskip
Clearly, $v_t$ is a continuous vector field on $D_t$ which is smooth and non-zero on $D_t \setminus Q_t$, zero on $Q_t$ and transversal to $\partial D_t$, pointing inwards. Moreover, each orbit associated to $v_t$ has a limit point in $Q_t$.

\medskip
So the flow $q_t: [0, \infty) \ \times \ (D_t \backslash Q_t) \to D_t \backslash Q_t$ associated to $v_1$ defines a continuous and surjective map:
$$
\begin{array}{cccc}
\xi'_t \ : & \! \partial D_t & \! \longrightarrow & \! Q_t \\
& \! u & \! \longmapsto & \! \displaystyle \lim_{\tau \to \infty} q_t(\tau,u)
\end{array}
.$$
\end{proof}

\begin{remark}
Lemma \ref{lemma_v_t} is still true if the set $Q_t$ is takes as the union of simple paths that intersect only at a point $\lambda_t$, instead of considering line segments. It is enough to consider an homeomorphism of the disk $D_t$ on itself which is a diffeomorphism outside the point $\lambda_t$.
\end{remark}

One has:

\begin{prop} \label{prop_vector_field}
We can choose a lifting of the vector field $v_t$ to a continuous vector field $E_t$ in $X_t$ so that: 
\begin{enumerate}
\item It is tangent to each stratum of the interior of $X_t$ (see Lemma \ref{lemma_strata});
\item It is smooth on each stratum of the induced stratification of $X_t \backslash P_t$;
\item It is integrable on $X_t \backslash P_t$;
\item It is zero on $P_t$;
\item It is transversal to $\partial X_t$ and points inwards;
\item The flow $\tilde{q}_t: [0, \infty) \ \times \ (X_t \backslash P_t) \to X_t$ associated to $E_t$ defines a map:
$$
\begin{array}{cccc}
\tilde{\xi}_t \ : & \! \partial X_t & \! \longrightarrow & \! P_t \\
& \! z & \! \longmapsto & \! \displaystyle \lim_{\tau \to \infty} \tilde{q}_t(\tau,z)
\end{array}
$$
such that $\tilde{\xi}_t$ is continuous, simplicial and surjective; 
\item The fiber $X_t$ is homeomorphic to the mapping cylinder of $\tilde{\xi}_t$.
\end{enumerate}
\end{prop}

\begin{proof}
Recall from Proposition \ref{prop_discriminant} that the restriction $\ell_t$ of the linear form $\ell$ to the Milnor fiber $X_t$ induces a stratified submersion: 
$$\varphi_t: X_t \setminus \ell_t^{-1} \big( \{y_1(t), \dots, y_k(t) \} \big) \to D_t \setminus \{y_1(t), \dots, y_k(t) \} \, .$$ 
So we can lift the vector field $v_t$ to a continuous vector field $E_t$ in $X_t$ that satisfies properties (1) to (5). See Lemma \ref{lemma_Ve}.

\medskip
Let us show that we can choose $E_t$ satisfying also (6). Fix $z \in \partial X_t$. We want to show that $\lim_{\tau \to \infty} \tilde{q}_t(\tau,z)$ exists, that is, that there exists a point $\tilde{p} \in P_t$ such that for any open neighborhood $\tilde{U}$ of $\tilde{p}$ in $X_t$ there exists $\theta>0$ such that $\tau>\theta$ implies that $\tilde{q}_t(\tau,z) \in \tilde{U}$. 

\medskip
From Lemma \ref{lemma_v_t} we know that there exists $p \in Q_t$ such that $\lim_{\tau \to \infty} q_t(\tau,\ell_t(z)) = p$, where $q_t: [0, \infty) \ \times \ (D_t \backslash Q_t) \to D_t$ is the flow associated to the vector field $v_t$. So for any small open neighborhood $U$ of $p$ in $D_t$ there exists $\theta>0$ such that $\tau>\theta$ implies that $q_t(\tau,\ell_t(z)) \in U$. Setting $\{ \tilde{p}_1, \dots, \tilde{p}_r \} := \ell_t^{-1}(p)$, we can consider $U$ sufficiently small such that there are disjoint connected components $\tilde{U}_1, \dots, \tilde{U}_r$ of $\ell_t^{-1}(U)$ such that each $\tilde{U}_j$ contains $\tilde{p}_j$. 

\medskip
Since $E_t$ is a lifting of $v_t$, we have that $q_t(\tau,\ell_t(z)) = \ell_t(\tilde{q}_t(\tau,z))$ for any $\tau \geq 0$. So $\tau>\theta$ implies that $\ell_t^{-1} \big( \ell_t(\tilde{q}_t(\tau,z)) \big) \subset \ell_t^{-1}(U)$. Hence for some $j \in \{1, \dots, r\}$ we have that $\tilde{q}_t(\tau,z) \in \tilde{U}_j$. Therefore $\lim_{\tau \to \infty} \tilde{q}_t(\tau,z) = \tilde{p}_j$. This proves $(6)$.

\medskip
Now we show that $X_t$ is homeomorphic to the mapping cylinder of $\tilde{\xi}_t$. In fact, the integration of the vector field $E_t$ on $X_t \backslash P_t$ gives a surjective continuous map:
$$\a: [0, \infty] \ \times \ \partial X_t \to X_t$$
that restricts to a homeomorphism:
$$\a_|: [0, \infty) \ \times \ \partial X_t \to X_t \backslash P_t \, .$$
Since the restriction $\a_{\infty}: \{\infty\} \times \partial X_t \to P_t$ is equal to $\tilde{\xi}_t$, which is surjective, it follows that the induced map: 
$$[\a_\infty]:  \big( (\{\infty\} \times \partial X_t ) \big/ \sim \big)  \to P_t$$
is a homeomorphism, where $\sim$ is the equivalent relation given by identifying $(\infty, z) \sim (\infty,z')$ if $\a_\infty(z) = \a_\infty(z')$.
Hence the map:
$$[\a]:  \big( ([0,\infty] \times \partial X_t) \big/ \sim \big)  \to X_t$$
induced by $\a$ defines a homeomorphism between $X_t$ and the mapping cylinder of $\tilde{\xi}_t$. This proves $(7)$.

\end{proof}

\medskip
\subsection{Second step: constructing the collapsing map}
\label{subsection_ss}
\ \\

First, let us recall that $X_t := \B_\e \cap \phi^{-1}(\BD_{\eta_1} \times \{t\})$ and that the map: 
$$\phi = (\ell,f): X \to \bC^2$$ 
induces a stratified submersion $\phi_|: \B_\e \cap X \cap \phi^{-1}(\BD_{\eta_1} \times \BD_{\eta_2} \setminus \Delta) \to \BD_{\eta_1} \times \BD_{\eta_2} \setminus \Delta$ (see Proposition \ref{prop_discriminant}). 

\medskip
Let $\gamma$ be a simple path in $\BD_{\eta_2}$ joining $0$ and some $t_0 \in \partial \BD_{\eta_2}$, such that $\gamma$ is transverse to $\partial \BD_{\eta_2}$. We want to describe the collapsing of $f$ along $\gamma$, that is, how $X_t$ degenerates to $X_0$ as $t \in \gamma$ goes to $0$.

\medskip
The first step is to construct the vanishing polyhedron $P_t$ simultaneously, for all $t$ in $\gamma$.

\medskip
Recall from Lemma \ref{lemma_polar_curve} that the polar curve $\Gamma$ is not contained in $f^{-1}(0)$. Hence $\Delta$ is not contained in $ \BD_{\eta_1} \times \{0\}$, and so the natural projection $\pi: \BD_{\eta_1} \times \BD_{\eta_2} \to  \BD_{\eta_2}$ restricted to $\Delta$ induces a ramified covering:
$${\pi}_|: \Delta \to \BD_{\eta_2} $$
of degree $k$, whose ramification locus is $\{0\} \subset \Delta$. Hence the inverse image of $\gamma \backslash \{0\}$ by this covering defines $k$ disjoint simple paths in $\Delta$, and each one of them is diffeomorphic to $\gamma \backslash \{0\}$. 

\medskip
Let $\Lambda$ be a simple path in $\BD_{\eta_1} \times \gamma$ such that $\pi(\Lambda) = \gamma$ and such that $\Lambda \cap \Delta = \{0\}$. Recall that $\Delta \cap (\BD_{\eta_1} \times \{t\}) = \{y_1(t), \dots, y_k(t)\}$. Hence for each $t \in \gamma$ the point $\lambda_t = \Lambda \cap (\BD_{\eta_1} \times \{t\})$ is a point in $(\BD_{\eta_1} \times \{t\}) \backslash \{y_1(t), \dots, y_k(t)\}$.

\medskip
We can choose the simple paths $\delta(y_j(t))$ joining $\lambda_t$ and $y_j(t)$, for each $j = 1, \dots, k$, in such a way that the set:
$$T_j := \bigcup_{t \in \gamma} \delta(y_j(t)) \,,$$
forms a triangle and $T_j \backslash \{0\}$ is differentially immersed in $\BD_{\eta_1} \times \gamma$, for each $j \in \{1, \dots, k\}$.
The intersection of any two triangles $T_j$ and $T_{j'}$ is the path $\Lambda$, for $j,j' \in \{1, \dots, k \}$ with $j \neq j'$. 

\medskip
Set $Q := \displaystyle\bigcup_{j=1}^k T_j$ and $P_\gamma := \phi^{-1}(Q)$, which we call a {\it collapsing cone for $f$ along $\gamma$}. See Figure \ref{Fig3}.

\begin{figure}[!h]
\centering
\includegraphics[scale=0.7]{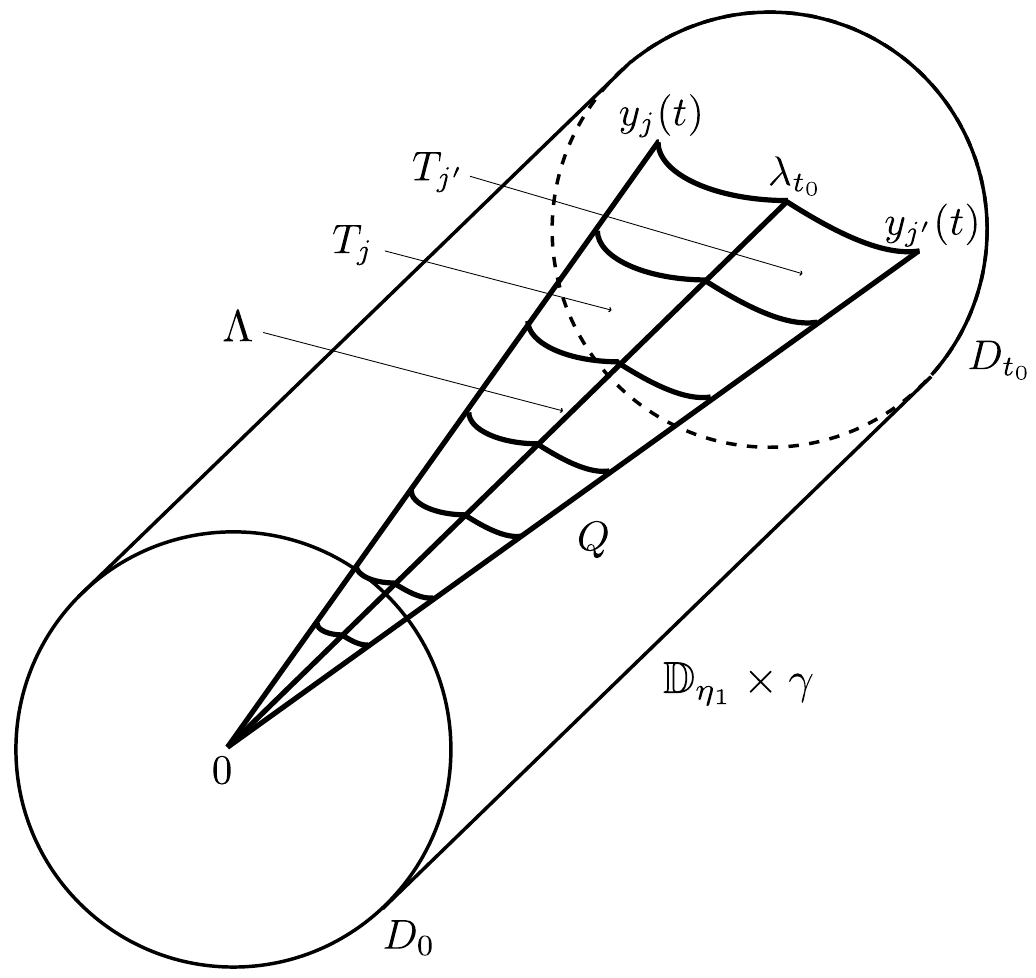}
\caption{}
\label{Fig3}
\end{figure}

\medskip
So $P_\gamma$ is a polyhedron adapted to the stratification induced by $\mathcal{S}$ such that the intersection $P_\gamma \cap X_t$ is a vanishing polyhedron $P_t$, for any $t \in \gamma \backslash \{0\}$, and $P_\gamma \cap X_0 = \{0\}$.

\medskip
\medskip
The second step is to construct a system of closed neighborhoods $V_A(Q)$ of $P_\gamma$ in $\BD_{\eta_1} \times \gamma$, for any $A \geq 0$, such that:
\begin{itemize}
\item[$(i)$] The boundary $\partial V_A(Q)$ of $V_A(Q)$ is smooth, for any $A \geq 0$;
\item[$(ii)$] $V_0(Q) = \BD_{\eta_1} \times \gamma$;
\item[$(iii)$] For any $A_1 > A_2$ one has $V_{A_1}(Q) \subset V_{A_2}(Q)$;
\item[$(iv)$] For any open neighborhood $U$ of $P_\gamma$ in $\BD_{\eta_1} \times \gamma$, there exists $A_U \geq 0$ sufficiently big such that $V_{A_U}(Q)$ is contained in $U$.
\end{itemize}
This system of neighborhoods will be used in the third step.

\medskip
To construct such a system of neighborhoods, we first consider a vector field $\mathcal{V}$ in $\BD_{\eta_1} \times \gamma$ that deformation retracts $\BD_{\eta_1} \times \gamma$ onto $Q$, just like in Lemma \ref{lemma_v_t}. That is, we consider a continuous vector field $\mathcal{V}$ in $\BD_{\eta_1} \times \gamma$ such that:
\begin{itemize}
\item[$\bullet$] It is smooth (and hence integrable) on $(\BD_{\eta_1} \times \gamma) \setminus Q$;
\item[$\bullet$] It is zero on $Q$;
\item[$\bullet$] It is transversal to $\partial \BD_{\eta_1} \times \gamma$; 
\item[$\bullet$] The projection of $\mathcal{V}$ onto $\gamma$ is zero.
\end{itemize}

\medskip
Just as we did in Proposition \ref{prop_vector_field}, we can use Proposition \ref{prop_discriminant} to obtain:

\begin{prop}
The vector field $\mathcal{V}$ can be lifted to a continuous vector field $E_\gamma$ in $X_\gamma$ such that:
\begin{itemize}
\item[$(i)$] For any $t \in \gamma \backslash \{0\}$, the restriction of $E_\gamma$ to $X_t$ gives a vector field $E_t$ as in the proposition above, relatively to the polyhedron $P_t = P_\gamma \cap X_t$ ;
\item[$(ii)$] The vector field $E_\gamma$ is integrable, stratified and non-zero on $X_\gamma \backslash P_\gamma$, tangent to each stratum of the interior of $X_\gamma \backslash P_\gamma$, zero on $P_\gamma$ and transversal to $\partial X_\gamma$, points inwards;
\item[$(iii)$] The flow $w: [0, \infty) \ \times \ \big( (\BD_{\eta_1} \times \gamma) \backslash Q \big) \to \BD_{\eta_1} \times \gamma$ associated to $\mathcal{V}$ defines a map:
$$ 
\begin{array}{cccc}
\xi \ : & \! \partial \BD_{\eta_1} \times \gamma & \! \longrightarrow & \! Q \\
& \! z & \! \longmapsto & \! \displaystyle \lim_{\tau \to \infty} w(\tau,z)
\end{array} 
$$
that is continuous, simplicial and surjective.
\end{itemize}
\end{prop}

\medskip
Now, for any real number $A>0$ set:
$$V_A(Q):=  (\BD_{\eta_1} \times \gamma) \backslash w \big( [0,A) \times \partial \BD_{\eta_1} \times \gamma \big) \, ,$$
which is a closed neighborhood of $Q$ in $\BD_{\eta_1} \times \gamma$. 

\medskip
Notice that $\partial V_A(Q)$ is a smooth manifold that fibers over $\gamma$ with fiber a circle, by the restriction of the projection $\pi: \BD_{\eta_1} \times \BD_{\eta_2} \to \BD_{\eta_2}$. Moreover, $\BD_{\eta_1} \times \gamma$ is clearly the mapping cylinder of $\xi$.

\medskip
\medskip
The third step is to construct a stratified and integrable vector field $\calE$ in  $X_\gamma  \backslash P_\gamma$, whose flow gives the degeneration of $X_{t_0}$ to $X_0$.

\medskip
Since we are assuming that $X$ has dimension two and since the restriction: 
$$\phi_|: \phi^{-1}\big( (\BD_{\eta_1} \times \BD_{\eta_2}) \backslash Q \big) \to (\BD_{\eta_1} \times \BD_{\eta_2}) \backslash Q$$
is a smooth covering, it follows that $\phi^{-1}\big( \partial V_A(Q) \big) \cap \B_\e$ is a smooth submanifold of $\phi^{-1}(\BD_{\eta_1} \times \BD_{\eta_2}) \cap \B_\e$ that is a proper locally trivial fibration over $\gamma$.

\medskip
Let $\theta$ be a vector field on $\gamma$ that goes from $t_0$ to $0$ in time $a>0$ and fix $A>0$. We are going to construct a smooth and integrable vector field $\calE$ in $X_\gamma \backslash P_\gamma$ that lifts $\theta$ and such that $\calE$ is tangent to $\phi^{-1} \big( \partial V_A'(Q) \big)$, for any$A' \geq A$. We will construct it locally, that is, for each point $p \in X_\gamma \backslash P_\gamma$ we will construct a vector field $\calE_p$ in some neighborhood $U_p$ of $p$, and then we will glue all of them using a partition of unity associated to the covering given by the neighborhoods $U_p$ (for the proof of the existence of a partition of unity associated to an infinite covering, see Lemma 41.6 of \cite{Mu}, for instance).

\medskip
Each $\calE_p$ is constructed in the following way:

\begin{itemize}
\item[$(a)$] If $p \notin \phi^{-1}\big( V_A(Q) \big) \cap \B_\e$, there is an open neighborhood $U_p$ of $p$ in $X_\gamma$ that does not intersect the closed set $\phi^{-1}\big( V_A(Q) \big) \cap \B_\e$. Then we define a smooth vector field $\calE_p$ on $U_p$ that lifts $\theta$.
\item[$(b)$] If $p \in \left[ \phi^{-1}\big( V_A(Q) \big) \cap \B_\e \right] \setminus P_\gamma$, there is an open neighborhood $U_p$ of $p$ in $X_\gamma$ that does not intersect $P_\gamma$. 
We define a smooth vector field $\calE_p$ on $U_p$ that lifts $\theta$ and that is tangent to $\phi^{-1}\big( \partial V_{A'}(Q) \big)$, for any $A' \geq A$.
\end{itemize}

\medskip
Then, as we said before, the vector field $\calE$ is obtained by gluing the vector fields $\calE_p$ using a partition of unity. Notice that $\calE$ lifts $\theta$.

\medskip
Hence the flow $h: [0,a] \times X_\gamma \backslash P_\gamma \to X_\gamma \backslash P_\gamma$ associated to $\calE$ defines a stratified homeomorphism $\Psi$ from $X_{t_0} \backslash P_{t_0}$ to $X_0 \backslash \{0\}$ that extends to a continuous map from $X_{t_0}$ to $X_0$ and that sends $P_{t_0}$ to $\{0\}$. 

\medskip
So we have proved that if $(X,0)$ is a $2$-dimensional reduced equidimensional germ of complex analytic space and if $f: (X,0) \to (\bC,0)$ has an isolated singularity at $0$, then there exist:

\begin{itemize}
\item[$(i)$] a vanishing polyhedron $P_t$ of real dimension one in the Milnor fiber $X_t$ of $f$ such that $X_t$ is a regular neighborhood of $P_t$ ;
\item[$(ii)$] a continuous map $\Psi_t: X_t \to X_0$ that sends $P_t$ to $\{0\}$ and that restricts to a homeomorphism $X_t \backslash P_t \to X_0 \backslash \{0\}$.
\end{itemize}

\medskip
\section{Proof of the main theorem}
\label{section_main_theorem}

Now we go back to the general case of a germ of complex analytic function 
$$f: (X,x) \to (\bC,0)$$ 
at a point $x$ of a reduced equidimensional complex analytic space $X \subset \bC^N$ of any dimension. Let ${\mathcal S} = (S_\a)_{\a \in A}$ be a Whitney stratification of $X$ and suppose that $f$ has an isolated singularity at $x$ in the stratified sense. 

\medskip
In order to simplify the notations, suppose further that $x$ is the origin in $\bC^N$.

\medskip
We observe that in the following sections $\Gamma$ denotes the polar curve of $f$ relatively to a generic linear form $\ell$ and that $\Delta$ denotes the polar discriminant of $f$ relatively to $\ell$, that is, $\Delta := \phi(\Gamma)$ where $\phi$ is the stratified map:
$$\phi := (\ell,f): (X,0) \to (\bC^2,0) \, .$$ 

As before, we assume that the polar curve $\Gamma$ is non-empty.

\medskip
In the next section, we will prove the following propositions:

\begin{prop} \label{prop_a}
For any $t \in \BD_{\eta_2}^*$ there exist:
\begin{itemize}
\vspace{0.17cm}
\item[$(i)$] A polyhedron $P_t$ in the Milnor fiber $X_t := \B_\e(x) \cap f^{-1}(t)$, with dimension $\dim_\bC X -1$, adapted to the stratification of $X_t$ induced by $\mathcal{S}$, i.e., the interior of each simplex of $P_t$ is contained in a stratum of $\mathcal{S}$.
\item[$(ii)$] A continuous vector field $E_t$ in $X_t$, tangent to each stratum of the interior of $X_t$, so that:

\begin{enumerate}
\item It is smooth on each stratum of the induced stratification of $X_t \setminus P_t$;
\item It is integrable on $X_t \setminus P_t$;
\item It is zero on $P_t$;
\item It is transversal to $\partial X_t$ (in the stratified sense) and points inwards;
\item The flow $\tilde{q}_t: [0, \infty) \ \times \ (X_t \backslash P_t) \to X_t$ associated to $E_t$ defines a map:
$$
\begin{array}{cccc}
\tilde{\xi}_t \ : & \! \partial X_t & \! \longrightarrow & \! P_t \\
& \! z & \! \longmapsto & \! \displaystyle \lim_{\tau \to \infty} \tilde{q}_t(\tau,z)
\end{array}
$$
such that $\tilde{\xi}_t$ is continuous, stratified, simplicial and surjective; 
\item The Milnor fiber $X_t$ is homeomorphic to the mapping cylinder of $\tilde{\xi}_t$.
\end{enumerate}
\end{itemize}
\end{prop}

\medskip
We say that the polyhedron $P_t$ above is a {\it vanishing polyhedron for $f$}.

\medskip
The idea of the construction of $P_t$ is quite simple and we will briefly describe it here. First recall the stratified map $\ell_t: X_t \to D_t$ given by the restriction of $\phi$ to $X_t$.

\medskip
By induction hypothesis, we have a vanishing polyhedron $P_t'$ for the restriction of $f$ to the hyperplane section $X \cap \{ \ell=0 \}$. 

\medskip
For each point $y_j(t)$ in the intersection of the polar discriminant $\Delta$ with the disk $D_t := \BD_{\eta_1} \times \{t\}$ as above, let $x_j(t)$ be a point in the intersection of the polar curve $\Gamma$ with $\ell_t^{-1}\big(y_j(t)\big)$. To simplify, we can assume that $x_j(t)$ is the only point in such intersection. 

\medskip
Also by the induction hypothesis, we have a collapsing cone $P_j$ for the restriction of the map $\ell_t$ to a small neighborhood of $x_j(t)$. The ``basis" of such cone is the polyhedron $P_j(a_j) := P_j \cap \ell_t^{-1}(a_j)$, where $a_j$ is a point in $\delta(y_j(t)) \backslash y_j(t)$ close to $y_j(t)$. 

\medskip
Since $\ell_t$ is a locally trivial fiber bundle over $\delta(y_j(t)) \backslash y_j(t)$, we can ``extend" the cone $P_j$ until it reaches the ``central" polyhedron $P_t'$. This gives a ``wing" that we denote by $C_j$. The union of all the wings $C_j$ together with $P_t'$ gives our vanishing polyhedron $P_t$.

\begin{figure}[!h] 
\centering 
\includegraphics[scale=0.3]{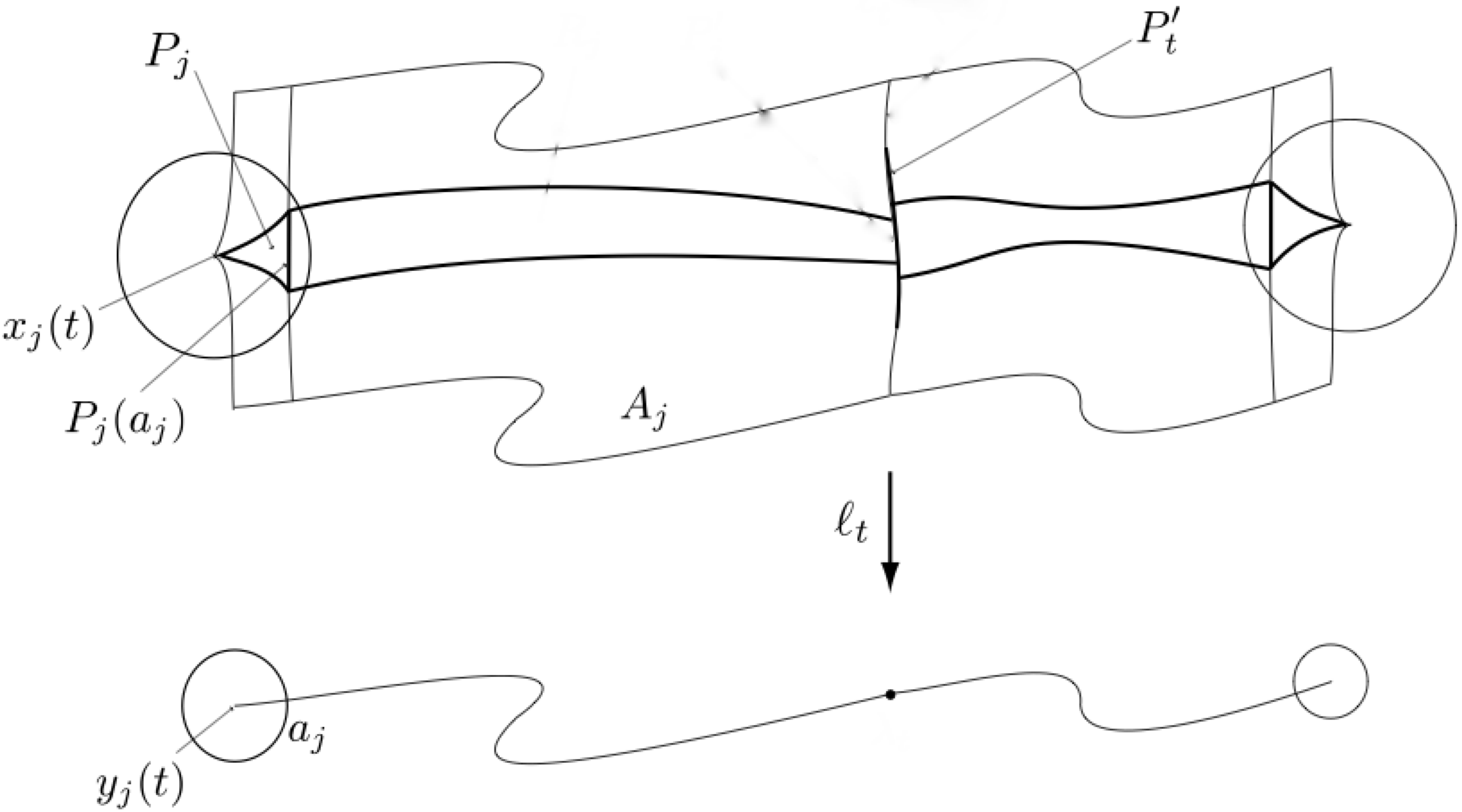}
\caption{}
\label{Fig5}
\end{figure}

\medskip
The detailed construction of $P_t$ is given in the next section, where we also show that the construction of the polyhedron $P_t$ and the vector field $E_t$ can be done simultaneously, for any $t$ along any simple path $\gamma$ in $\BD_{\eta_2}$ joining $0$ and some $t_0 \in \BD_{\eta_2}$.

\medskip
Precisely, we will prove the following:

\begin{prop} \label{prop_b}
Let $\gamma$ be a simple path in $\BD_{\eta_2}$ joining $0$ and some $t_0 \in \BD_{\eta_2}$. There exist a polyhedron $P_\gamma$ in $X_\gamma := X \cap f^{-1}(\gamma) \cap \B_\e$, adapted to the stratification induced by $\mathcal{S}$, and a continuous vector field $E$ in $X_\gamma$, tangent to each stratum of the interior of $X_\gamma$, such that:
\begin{itemize}
\vspace{0.17cm}
\item[$(i)$] the intersection $P_\gamma \cap X_t$ is a vanishing polyhedron $P_t$ as in the proposition above, for any $t \in \gamma \backslash \{0\}$, and $P_\gamma \cap X_0 = \{0\}$;
\item[$(ii)$] for any $t \in \gamma \backslash \{0\}$, the restriction of $E$ to $X_t$ gives a vector field $E_t$ as in the proposition above, relatively to the polyhedron $P_t = P_\gamma \cap X_t$;
\item[$(iii)$] the vector field $E$ is integrable on $X_\gamma \setminus P_\gamma$, smooth on each stratum of the induced stratification of $X_\gamma \setminus P_\gamma$, non-zero outside $P_\gamma$, zero on $P_\gamma$, transversal to $\partial X_\gamma := X_\gamma \cap \BS_\e$ in the stratified sense, and points inwards.
\end{itemize}
\end{prop}

\medskip
We say that the polyhedron $P_\gamma$ above is a {\it collapsing cone for $f$ along the path $\gamma$}.

\medskip
One can check that the flow $\tilde{q}: [0,\infty) \times (X_\gamma \backslash P_\gamma) \to X_\gamma$ given by the integration of the vector field $E$ on $X_\gamma \setminus P_\gamma$ defines a continuous, simplicial and surjective map: 
$$ 
\begin{array}{cccc}
\tilde{\xi} \ : & \! \partial X_\gamma & \! \longrightarrow & \! P_\gamma \\
& \! z & \! \longmapsto & \! \displaystyle \lim_{\tau \to \infty} \tilde{q}(\tau,z)
\end{array} 
$$
such that $X_\gamma$ is homeomorphic to the mapping cylinder of $\tilde{\xi}$ (see Proposition \ref{prop_vector_field}). 

\medskip
Assuming that Propositions \ref{prop_a} and \ref{prop_b} are true, we can easily prove Theorem \ref{theo_main} of the Introduction as follows:

\medskip
Fix $t \in \BD_{\eta_2}^*$ and let $\gamma$ be a simple path in $\BD_{\eta_2}$ connecting $t$ and $0$. Consider the polyhedron $P_\gamma$ and the vector field $E$ in $X_\gamma$ given by Propositions \ref{prop_a} and \ref{prop_b}, as well as the flow $\tilde{q}$ given by the integration of $E$. 

\medskip
For any positive real $A>0$ set:
$$\tilde{V}_A(P_\gamma):= X_\gamma \setminus \tilde{q} \big( [0,A) \times \partial X_\gamma \big) \, ,$$
which is a closed neighborhood of $P_\gamma$ in $X_\gamma$. Notice that using the first isotopy theorem of Thom-Mather, the boundary $\partial \tilde{V}_A(P_\gamma)$ of $\tilde{V}_A(P_\gamma)$ is a locally trivial topological fibration over $\gamma$.

\medskip
Following the steps $(a)$ and $(b)$ of the end of section \ref{subsection_ss} and using the first isotopy theorem of Thom-Mather, we can construct a similar continuous vector field $\calE$ on $X_\gamma \backslash P_\gamma$ such that:
\begin{itemize}
\item[$\circ$] it lifts a smooth vector field $\theta$ on $\gamma$ that goes from $t_0$ to $0$ in a time $a>0$;
\item[$\circ$] it is smooth on each stratum of $X_\gamma \backslash P_\gamma$;
\item[$\circ$] it is rugose and hence integrable;
\item[$\circ$] it is tangent to $ \partial \tilde{V}_A(P_\gamma) $, for any $A > 0$.
\end{itemize}

\medskip
In fact, the vector fields constructed here are rugose because under the hypothesis of the first isotopy theorem of Thom-Mather, a smooth vector field lifts to a rugose vector field. See Lemma \ref{lemma_Ve}.

\medskip
So the flow $g: [0,a] \times X_\gamma \backslash P_\gamma \to X_\gamma \backslash P_\gamma$ associated to $\calE$ defines a homeomorphism $\Psi_t$ from $X_t \backslash P_t$ to $X_0 \backslash \{0\}$ that extends to a continuous map from $X_t$ to $X_0$ and that sends $P_t$ to $\{0\}$, for any $t \in \gamma \backslash \{0\}$. This proves Theorem \ref{theo_main}.

\medskip
\section{Proof of Propositions \ref{prop_a} and \ref{prop_b}}
\label{section_prop_a}

In this section, we shall prove Propositions \ref{prop_a} and \ref{prop_b}. 

\medskip
Actually, the same arguments used in the proof of Proposition \ref{prop_b} give a proof for the following stronger proposition, where instead of constructing the polyhedron $P_t$ and the vector field $E_t$ simultaneously for $t$ in a path $\gamma$, we construct them simultaneously for $t$ in a closed semi-disk $\BD^+$ of $\BD_{\eta_2}$.

\begin{prop} \label{prop_b2}
Let $\BD^+$ be a closed semi-disk in $\BD_{\eta_2}$. There exist a polyhedron $P^+$, adapted to the stratification induced by $\mathcal{S}$, and a continuous vector field $E$ in $X^+ := X \cap f^{-1}(\BD^+) \cap \B_\e$, tangent to each stratum of the interior of $X^+$, such that:
\begin{itemize}
\vspace{0.17cm}
\item[$(i)$] the intersection $P^+ \cap X_t$ is a vanishing polyhedron $P_t$ as in Proposition \ref{prop_a}, for any $t \in \BD^+ \backslash \{0\}$, and $P^+ \cap X_0 = \{0\}$;
\item[$(ii)$] for any $t \in \BD^+ \backslash \{0\}$, the restriction of $E$ to $X_t$ gives a vector field $E_t$ as in Proposition \ref{prop_a}, relatively to the polyhedron $P_t = P^+ \cap X_t$;
\item[$(iii)$] the vector field $E$ is integrable on $X^+ \backslash P^+$, smooth on each stratum of the induced stratification of $X^+ \backslash P^+$, non-zero outside $P^+$, zero on $P^+$, transversal to $\partial X^+ := X^+ \cap \BS_\e$ in the stratified sense, and points inwards.
\end{itemize}
\end{prop}

\medskip
We say that the polyhedron $P^+$ is a {\it collapsing cone for $f$ along the semi-disk $\BD^+$}.

\medskip
We will prove Propositions \ref{prop_a} and \ref{prop_b} by finite induction on the dimension $n$ of $X$. Proposition \ref{prop_b2} will be used (as induction hypothesis) in the construction of the vector field $E_t$ of Proposition \ref{prop_a}.

\medskip
We have already proved Propositions \ref{prop_a} and \ref{prop_b} when $n=2$, and one can check that it is easy to extend it to a proof of Proposition \ref{prop_b2} in the two-dimensional case. Now we will prove that if these results are true whenever $\dim_\bC X \leq n-1$, then they are true when $\dim_\bC X =n$.

\medskip
\subsection{Proof of Proposition \ref{prop_a}: constructing the vanishing polyhedron}
\label{subsection_CLP2}
\ \\

As we said above, the polyhedron $P_t$ will consist of a ``central'' polyhedron $P_t'$ on which we will attach ``wings'' $C_j$. The first step will be to construct the central polyhedron $P_t'$. Then we will construct the extremity of each wing $C_j$. Finally, we will extend the extremity of each wing until it hits $P_t'$.

\medskip
Recall that we have fixed a linear form $\ell: \bC^N \to \bC$ that takes $0 \in \bC^N$ to $0 \in \bC$ and that satisfies the conditions of Lemma \ref{lemma_polar_curve}. Then $\Gamma$ is the polar curve of $f$ relatively to $\ell$ at $0$ and $\Delta$ is the polar discriminant of $f$ relatively to $\ell$ at $0$. 

\medskip
Also recall from Proposition \ref{prop_discriminant} that the map $\phi=(\ell,f)$ induces a stratified submersion: 
$$\phi_|: \B_\e \cap X \cap \phi^{-1}(\BD_{\eta_1} \times \BD_{\eta_2} \setminus \Delta) \to \BD_{\eta_1} \times \BD_{\eta_2} \setminus \Delta$$
and that for each $t \in \BD_{\eta_2}^*$ fixed, the restriction $\ell_t$ of $\ell$ to the Milnor fiber $X_t$ induces a topological locally trivial fibration:
$$\varphi_t: X_t \backslash \ell_t^{-1}\big(\{y_1(t), \dots, y_k(t)\}\big) \to D_t \backslash \{y_1(t), \dots, y_k(t)\} \, ,$$
where $D_t = \BD_{\eta_1} \times \{t\}$ and $\{y_1(t), \dots, y_k(t)\} = \Delta \cap D_t$.

\medskip
For any $t \in \BD_{\eta_2}$ set $\lambda_t := (0,t)$. Since the complex line $\{0\} \times \bC$ is not a component of $\Delta$, we can suppose that $\lambda_t \notin \{y_1(t), \dots, y_k(t)\}$. 

\medskip
For each $j = 1, \dots, k$, let $\delta(y_j(t))$ be a simple path in $D_t$ starting at $\lambda_t$ and ending at $y_j(t)$, such that two of them intersect only at $\lambda_t$. 

\medskip
\underline{First step: constructing the central polyhedron $P_t'$:}
\medskip

Consider $f'$ the restriction of $f$ to the intersection $X \cap \{\ell=0\}$, which has complex dimension $n-1$. Then we can apply the induction hypothesis to $f'$ to obtain a vanishing polyhedron $P_t'$ in the fiber $X_t \cap \{ \ell=0 \}$ and a integrable vector field $E_t'$ in $X_t \cap \{ \ell=0 \}$, stratified on $(X_t \cap \{ \ell=0 \}) \setminus P_t'$, that deformation retracts $X_t \cap \{ \ell=0 \}$ onto $P_t'$.

\medskip
\underline{Second step: constructing the extremity of the wings $C_j$:}
\medskip

First of all, in order to make it easier for the reader to understand the constructions, we will suppose that $\Gamma$ intersects $\ell_t^{-1}\big(y_j(t)\big)$ in only one point, which we call $x_j(t)$. The proof of the general case follows the same steps. In fact, we do the following conjecture:

\begin{conj} 
For $\ell$ general enough, the map-germ $\phi_\ell = (\ell,f): (X,x) \to (\bC^2,0)$ induces a bijective morphism between $\Gamma$ and $\Delta$. 
\end{conj}

Now recall that $\ell_t$ induces a locally trivial fibration over $\delta(y_j(t)) \backslash \{y_j(t)\}$. If we look at the local situation at $x_j(t)$, we can apply the induction hypothesis to the germ ${\ell_t}_|: (X_t, x_j(t)) \to (D_t, y_j(t))$, which has an isolated singularity at $x_j(t)$ in the stratified sense, in lower dimension. That is, considering $B_j$ a small ball in $\bC^{N}$ centered at $x_j(t)$; $D_s$ a small disk in $D_t$ centered at $y_j(t)$ and $D_s^+$ a semi-disk of $D_s$ which contains $\delta(y_j(t)) \cap \mathring{D}_s$ in its interior, we obtain:
\begin{itemize}
\item a collapsing cone $P_j^+$ for $\ell_t$ along the semi-disk $D_s^+$;
\item a collapsing cone $P_j$ for $\ell_t$ along the path $D_s \cap \delta(y_j(t))$; and
\item a vector field $E_j$ in $\ell_t^{-1}(D_s) \cap B_j$; 
\end{itemize}
which give the collapsing of the map ${\ell_t}_|: B_j \cap \ell_t^{-1}(D_s) \to D_s $. See Figure \ref{Fig4}.

\begin{figure}[!h] 
\centering 
\includegraphics[scale=0.55]{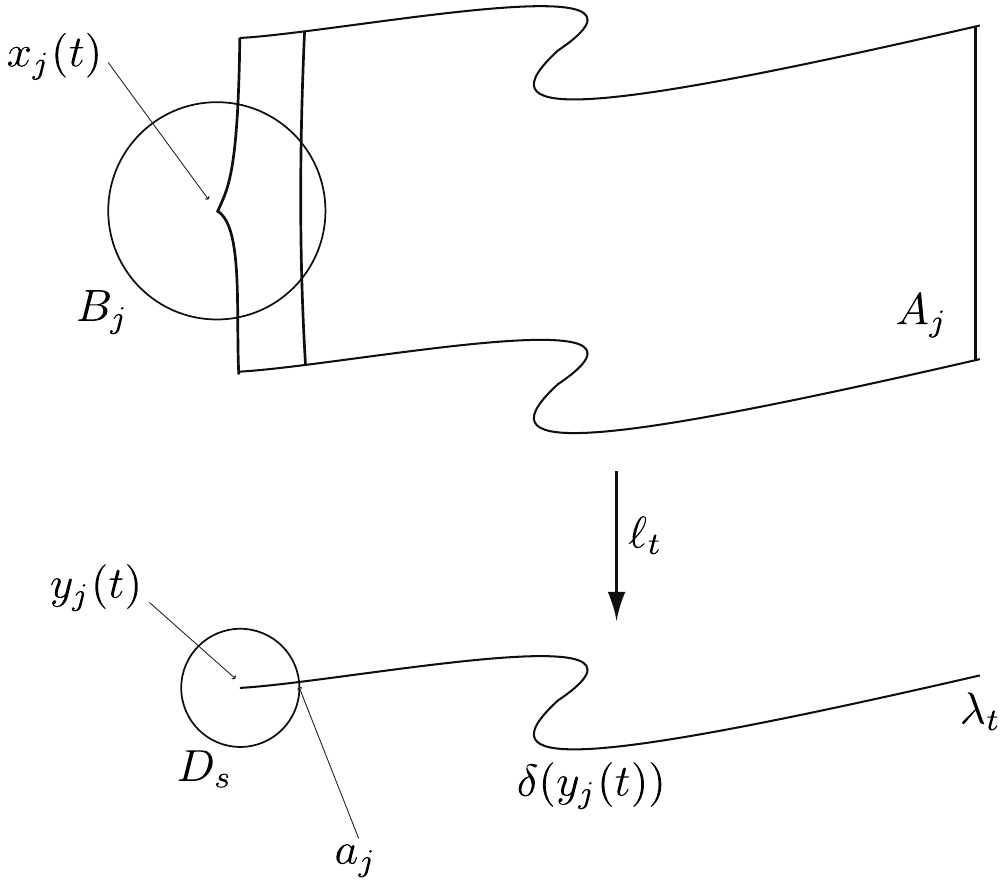}
\caption{}
\label{Fig4}
\end{figure}

\medskip
So the extremity of the wing $C_j$ is given by the collapsing cone $P_j$.

\medskip
\underline{Third step: extending the extremity of each wing until it hits $P_t'$:}
\medskip

First we need to construct the following vector fields on $A_j := \ell_t^{-1}\big(\delta(y_j(t)) \backslash \{y_j(t)\}\big)$ that will be used to extend the cone $P_j$ and glue it on the central polyhedron $P_t'$:

\medskip
\begin{itemize}
\item[$\bullet$] {\bf Vector Field $\Xi$:} Let $\xi$ be a smooth non-singular vector field on $\delta(y_j(t)) \backslash \{y_j(t)\}$ that goes from $y_j(t)$ to $\lambda_t = (0,t)$. Since the restriction of $\ell_t$ to each Whitney stratum $S_\a$ has maximum rank over $\delta(y_j(t)) \backslash \{y_j(t)\}$, we can lift $\xi$ to an integrable stratified vector field $\Xi$ on $A_j$ (see Lemma \ref{lemma_Ve}). In particular, for any $u \in \delta(y_j(t)) \backslash \{y_j(t)\}$ we can use the vector field $\Xi$ to obtain a stratified homeo\-morphism $\a_u: \ell_t^{-1}(\lambda_t) \to \ell_t^{-1}(u)$, which takes $P_t'$ to a polyhedron $\a_u(P_t')$ in $\ell_t^{-1}(u)$. 

\medskip
\item[$\bullet$] {\bf Vector Field $\calV$:} We can transport the vector field $E_t'$ of $\ell_t^{-1}(\lambda_t) = X_t \cap \{ \ell=0 \}$ given by the induction hypothesis to all the fibers $\ell_t^{-1}(u)$, for any $u \in \delta(y_j(t)) \backslash \{y_j(t)\}$. The transportation of $E_t'$ to $\ell_t^{-1}(u)$ is the vector field on $\ell_t^{-1}(u)$ given by the flow obtained as image by $\a_u$ of the flow given by $E_t'$. So we obtain a vector field $\calV$ on $A_j$ whose restriction to $\ell_t^{-1}(\lambda_t)$ is $E_t'$. The flow associated to $\calV$ takes a point $z \in \ell_t^{-1}(u)$ to the polyhedron $\a_u(P_t')$.

\medskip
\item[$\bullet$] {\bf Vector Field $\calV_1$:} Let $\theta$ be a smooth function on $\delta(y_j(t))$ such that $\theta(\lambda_t)=0$ and such that $\theta$ is non-singular and positive on $\delta(y_j(t)) \backslash \{\lambda_t \}$. It induces a function $\tilde{\theta} := \theta \circ \ell_t$ defined on $A_j$. Set:
$$\calV_1 := \calV + \tilde{\theta} \cdot \Xi \, ,$$
which is an integrable vector field, tangent to the strata of the interior of $A_j$ induced by $\mathcal{S}$. Furthermore, this vector field $\calV_1$ is pointing inwards on the boundary $\partial A_j$, i.e. transversal in $A_j$ to the strata of $\partial A_j$ induced by $\mathcal{S}$.
\end{itemize}

\medskip
Since the vectors $\vec \calV(z)$ and $\vec \Xi(z)$ are not parallel for any $z \in A_j \backslash P_t'$, the vector field $\calV_1$ is zero only on the vanishing polyhedron $P_t'$ of $\ell_t^{-1}(\lambda_t)$. Then if $z$ is a point in $A_j \backslash \ell_t^{-1}(\lambda_t)$, the orbit of $\calV_1$ that passes through $z$ has its limit point $z_1'$ in $P_t'$. 

\medskip
Moreover, since the orbit of $\calV$ that passes through $z$ has its limit point $z'$ in the transportation $\a_{\ell_t(z)}(P_t')$ of $P_t'$ to $\ell_t^{-1}\big(\ell_t(z)\big)$ by $\Xi$, it follows that $z_1'$ is the point corresponding to $z'$ by $\Xi$, that is, $z_1' = \a_{\ell_t(z)}(z')$. In fact, if $u:= \ell_t(z)$ and if $w := (\a_u)^{-1}(z)$ is the corresponding point in  $\ell_t^{-1}(\lambda_t)$, then by construction the integral curve $\mathcal{C}_\calV(z)$ associated to $\calV$ that contains $z$ is given by $\a_u(\mathcal{C}(w))$, where $\mathcal{C}(w)$ is the integral curve associated to $E_t'$ that contains $w$. 

\medskip
Set $a_j:= \partial D_s \cap \delta(y_j(t))$ and $P_j(a_j) := P_j \cap \ell_t^{-1}(a_j)$, where $P_j$ is the collapsing cone for $\ell_t$ at $x_j(t)$ along the path $D_s \cap \delta(y_j(t))$, as defined above. By the previous paragraph, $\calV_1$ takes $P_j(a_j)$ to $P_t'$.

\medskip
Since the action of the flow given by $\calV$ is simplicial, we can assume that the action of the flow given by $\calV_1$ is simplicial. Then the image of $P_j(a_j)$ by the flow of $\calV_1$ is a sub-polyhedron $P_j'$ of $P_t'$. Moreover, the orbits of the points in $P_j(a_j)$ give a polyhedron $R_j$. See Figure \ref{Fig5}. 

\begin{figure}[!h] 
\centering 
\includegraphics[scale=0.7]{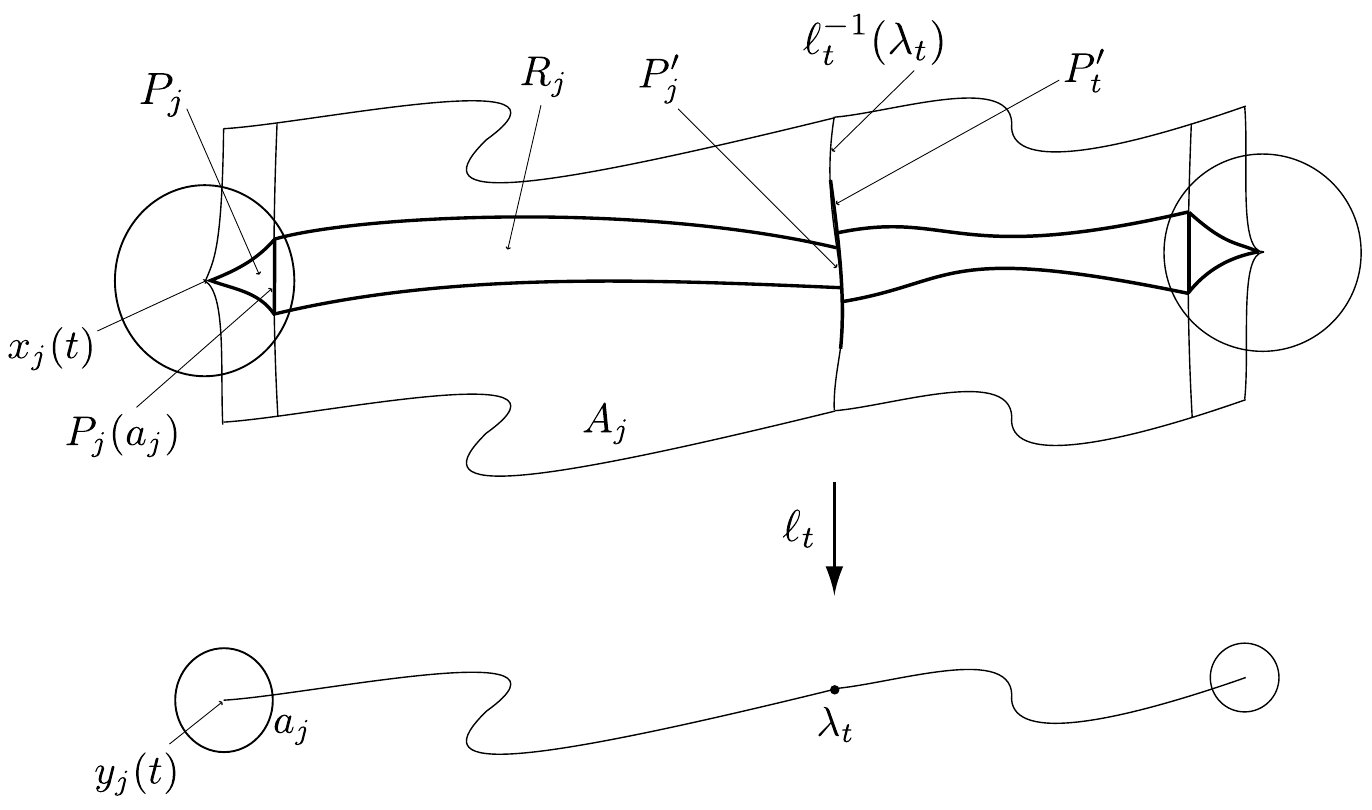}
\caption{}
\label{Fig5}
\end{figure}

\medskip
Set:
$$C_j := P_j \cup R_j \cup P_j' \, .$$
We call $C_j$ a {\it wing} of the polyhedron $P_t$. In the case when $X$ is smooth, it corresponds to a Lefschetz thimble.

\medskip
Then the polyhedron we are going to consider is: 
$$P_t := P_t' \bigcup_{j=1}^k C_j \, .$$
It is adapted to the stratification $\mathcal{S}$, since $P_j$ is adapted to $\mathcal{S}$ and the vector field $\mathcal{V}_1$ is tangent to the strata of $\mathcal{S}$.

\medskip
Now we have:

\begin{lemma} \label{prop_Et}
There is a continuous vector field $E_t$ on $X_t$ such that:
\begin{itemize}
\item[$(i)$] It is tangent to each stratum of $\mathring{X_t} \backslash P_t$ induced by $\mathcal{S}$, where $\mathring{X_t} := X_t \backslash \partial X_t$;
\item[$(ii)$] It is transversal to the strata of $\partial X_t$ and points inwards;
\item[$(iii)$] It is smooth on each stratum of $X_t \backslash P_t$;
\item[$(iv)$] It is integrable on $X_t \backslash P_t$;
\item[$(v)$] It is non-zero on $X_t \backslash P_t$ and it is zero on $P_t$;
\item[$(vi)$] The orbits of $E_t$ have a limit point at $P_t$ when the parameter goes to the infinity.
\end{itemize}
\end{lemma}

\medskip
The vector field $E_t$ is obtained by gluing several vector fields on $X_t$ given by the lifting by $\varphi_t$ of suitable vector fields on $D_t$. The detailed proof of Lemma \ref{prop_Et} is quite annoying since it contains too many technical steps and constructions, so we present it separately in the next section.

\medskip
The flow defined by the vector field $E_t$ of Lemma \ref{prop_Et} gives a continuous, surjective and simplicial map $\tilde{\xi}_t$ from $\partial X_t$ to $P_t$ such that $X_t$ is homeomorphic to the simplicial map cylinder of $\tilde{\xi}_t$ (see the proof of Proposition \ref{prop_vector_field}). This proves Proposition \ref{prop_a}.

\medskip
\subsection{Proof of Proposition \ref{prop_b}: constructing the collapsing cone}
\label{subsection_collapsing}
\ \\

Given a simple path $\gamma$ in $\BD_{\eta_2}$ joining $0 \in \BD_{\eta_2}$ and some $t_0 \in \BD_{\eta_2}$, we have to construct a polyhedron $P_\gamma$, adapted to the stratification induced by $\mathcal{S}$, and a continuous vector field $E$ in $X_\gamma := X \cap f^{-1}(\gamma) \cap \B_\e$, tangent to each stratum of $X_\gamma$, satisfying the conditions $(i)$, $(ii)$ and $(iii)$ of Proposition \ref{prop_b}. 

\medskip
We are going to construct $P_\gamma$. The construction of the vector field $E$ follows the same steps of the construction of the vector field $E_t$ of the subsection above, which is described with details in the next section.

\medskip
First, consider the polyhedron $P_{t_0}$, contained in the Milnor fiber: 
$$X_{t_0} = X \cap f^{-1}(t_0) \cap \B_\e$$
and constructed as in Proposition \ref{prop_a}. It is given by the union:
$$P_{t_0} = P_{t_0}' \bigcup_{j=1}^k C_j \, ,$$
where each $C_j$ is a wing glued to $P_{t_0}'$ along a sub-polyhedron $(P_j)_{t_0}'$ of $P_{t_0}'$ and $P_{t_0}'$ is the vanishing polyhedron for the restriction $f'$ of $f$ to $X \cap \{\ell=0\}$. 

\medskip
By the induction hypothesis, we also have a collapsing cone $P_\gamma'$ in $X_\gamma \cap \{\ell=0\}$ and a continuous vector field $G'$ in $X_\gamma \cap \{\ell=0\}$ that gives the degeneration of $f'$. That is, the restriction of $G'$ to $(X_\gamma \cap \{\ell=0\}) \backslash P_\gamma'$ is an integrable vector field whose associated flow defines a homeomorphism from $(X_{t_0} \cap \{\ell=0\}) \backslash P_{t_0}'$ to $(X_{t_0} \cap \{\ell=0\}) \backslash \{0\}$ that extends to a continuous map from $X_{t_0} \cap \{\ell=0\}$ to $X_0 \cap \{\ell=0\}$ and that sends $P_t'$ to $\{0\}$.

\medskip
Recall from Proposition \ref{prop_discriminant} that the map $\phi = (\ell,f)$ induces a stratified submersion: 
$$\phi_|: \B_\e \cap X \cap \phi^{-1}(\BD_{\eta_1} \times \BD_{\eta_2} \setminus \Delta) \to \BD_{\eta_1} \times \BD_{\eta_2} \setminus \Delta \, ,$$
where $\Delta \subset  \BD_{\eta_1} \times \BD_{\eta_2}$ is the polar discriminant of $f$ relatively to the linear form $\ell$.

\medskip
Set $\Lambda := \{0\} \times \gamma$, which we suppose that intersects $\Delta$ only at $0 \in \bC^2$. Then we define the $2$-dimensional polyhedra $T_j$ in $\BD_{\eta_1} \times \gamma$, for each $j=1, \dots, k$, as in subsection \ref{subsection_ss}. That is:
$$T_j := \bigcup_{t \in \gamma} \delta(y_j(t)) \, ,$$
where each $\delta(y_j(t))$ is a simple path connecting $y_j(t)$ and $\lambda_t= (0,t)$, and $\delta(y_j(0))$ is the origin.

\medskip
For each $x_j(t)$ over $y_j(t)$, with $t \in \gamma$, choose a small radius $r(t)$ such that the set: 
$$\calB_j := \bigcup_{t \in \gamma} \B_{r(t)}(x_j(t))$$
is a neighborhood of $\cup_{t \in \gamma^*}\{ x_j(t) \}$ conic from $0$, where $r(t)$ is a real analytic function of $t$ with $r(0)=0$ and $\gamma^* := \gamma \backslash \{0\}$.

\medskip
To each $\calB_j$ one can associate a neighborhood: 
$$\calA_j := \bigcup_{t \in \gamma} \BD_{s(t)}(y_j(t))$$
in $\BD_{\eta_1} \times \gamma$, where $s(t)$ is an analytic function of $t \in \gamma$ with $0<s(t) \ll r(t)$, if $t \neq 0$, and $s(0)=0$.

\medskip
Also let $\calU$ be a neighborhood of $\Lambda \backslash \{0\}$, conic from $0$, that meets all the $\calA_j$'s, but not containing any $y_j(t)$. See Figure \ref{Fig6}.

\begin{figure}[!h]
\centering
\includegraphics[scale=0.8]{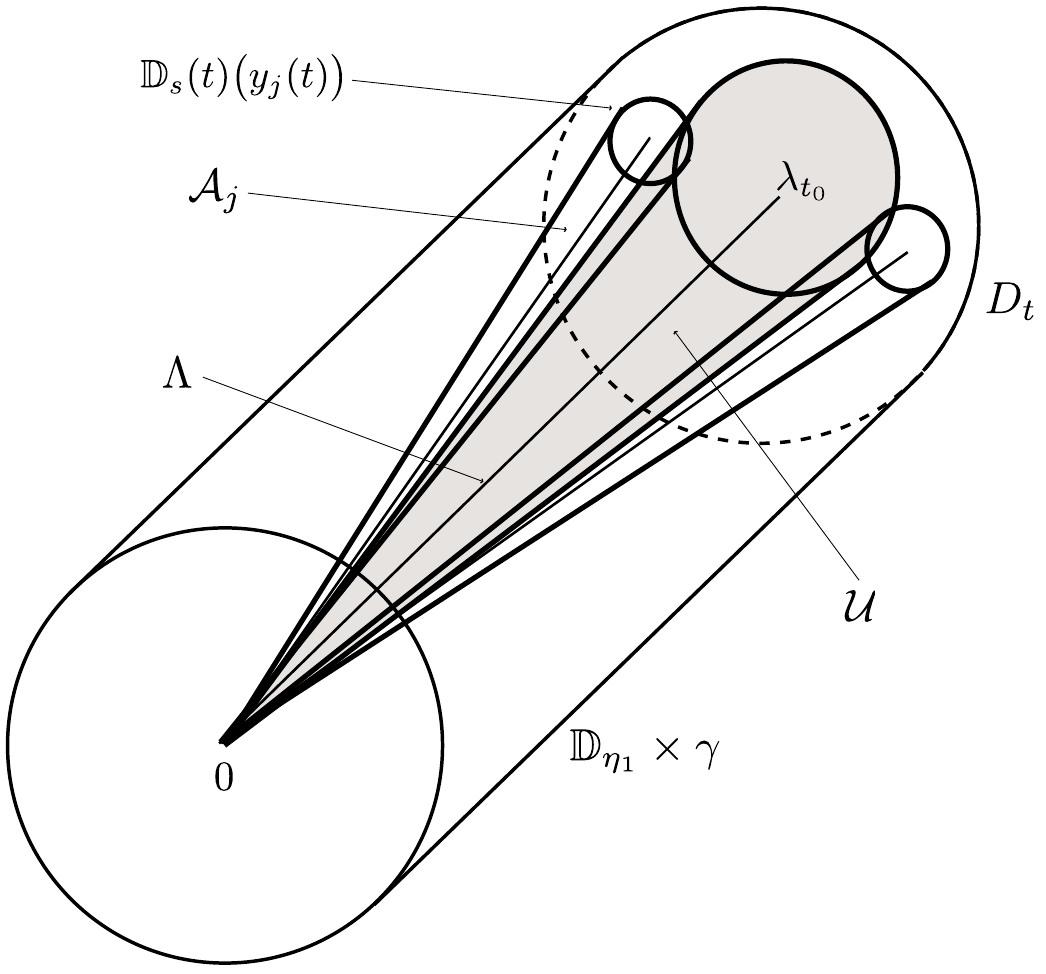}
\caption{}
\label{Fig6}
\end{figure}

\medskip
Notice that the stratified vector field $G'$ that gives the degeneration of the restriction of $f$ to $X \cap \{\ell=0\}$ is defined on $\phi^{-1}(\Lambda)$. Set:
$$\tilde{\calU} := \phi^{-1}(\calU) \cap X_\gamma \, . $$

\medskip
Since $\calU$ is contractible, we can extend the vector field $G'$ to a integrable stratified vector field $G_\calU$, defined on $\tilde{U}$. Notice that the flow given by the vector field $G_\calU$ sends the intersection $P_{t_0} \cap \tilde{\calU}$ to $0$, where $P_{t_0}$ is the polyhedron in $X_{t_0}$ previously constructed.

\medskip
One can also construct an integrable vector field $G_j$ on each $\calB_j \backslash \{0\}$ that trivializes it over $\gamma \backslash \{0\}$.

\medskip
Then, using a partition of unity $(\rho_\calU, \rho_1, \dots, \rho_k)$ adapted to $\tilde{\calU}, \calB_1, \dots, \calB_k$, we glue all the vector fields $G_j$'s and $G_\calU$ together. We obtain a continuous trivializing vector field: 
$$G := \rho_\calU G_\calU + \sum_{j=1}^k \rho_j G_j$$
in $\tilde{\calU} \cup_{j=1}^k \calB_j$ such that:
\begin{itemize}
\item[$\circ$] it is smooth on each stratum of $X_\gamma \cap (\tilde{\calU} \cup_{j=1}^k \calB_j)$ ;
\item[$\circ$] it is rugose and hence integrable;
\item[$\circ$] it projects to a radial vector field in $\gamma$ that converges to $0$. 
\end{itemize}

\medskip
Then we construct the vanishing cone $P_\gamma$ from the vanishing polyhedron $P_{t_0}$ using the flow of $G$.

\medskip
\section{Proof of Lemma \ref{prop_Et}: constructing the vector field $E_t$}
\label{section_prop_Et}

In this section we give the detailed construction of the vector field $E_t$ of Lemma \ref{prop_Et}, whose flow gives a continuous, surjective and simplicial map $\tilde{\xi}_t$ from the boundary of the Milnor fiber $\partial X_t$ to the polyhedron $P_t$ constructed in the previous section, such that the Milnor fiber $X_t := \phi^{-1}(\BD_{\eta_1} \times \{t\}) \cap \B_\e(x)$ is homeomorphic to the simplicial map cylinder of $\tilde{\xi}_t$.

\medskip
Recall that we have fixed a linear form $\ell: \bC^N \to \bC$ that takes $0 \in \bC^N$ to $0 \in \bC$ and that satisfies the conditions of Lemma \ref{lemma_polar_curve}. Then $\Gamma$ is the polar curve of $f$ relatively to $\ell$ at $0$ and $\Delta$ is the polar discriminant of $f$ relatively to $\ell$ at $0$. 

\medskip
Also recall from Proposition \ref{prop_discriminant} that the map $\phi=(\ell,f)$ induces a stratified submersion: 
$$\phi_|: \B_\e \cap X \cap \phi^{-1}(\BD_{\eta_1} \times \BD_{\eta_2} \setminus \Delta) \to \BD_{\eta_1} \times \BD_{\eta_2} \setminus \Delta$$
and that for each $t \in \BD_{\eta_2}^*$ fixed, the restriction $\ell_t$ of $\ell$ to the Milnor fiber $X_t$ induces a topological locally trivial fibration:
$$\varphi_t: X_t \backslash \ell_t^{-1}\big(\{y_1(t), \dots, y_k(t)\}\big) \to D_t \backslash \{y_1(t), \dots, y_k(t)\} \, ,$$
where $D_t := \BD_{\eta_1} \times \{t\}$ and $\{y_1(t), \dots, y_k(t)\} = \Delta \cap D_t$.

\medskip
As before, take a point $\lambda_t$ in $D_t$ such that $\lambda_t \notin \{y_1(t), \dots, y_k(t)\}$. Also, for each $j = 1, \dots, k$, let $\delta(y_j(t))$ be a simple path in $D_t$ starting at $\lambda_t$ and ending at $y_j(t)$, such that two of them intersect only at $\lambda_t$. We defined the set $Q_t := \cup_{j=1}^k \delta(y_j(t))$.

\medskip
Also recall that we can apply the induction hypothesis to the restriction $f'$ of $f$ to the intersection $X \cap \{\ell=0\}$, which has complex dimension $n-1$. We obtain a vanishing polyhedron $P_t'$ in the intersection $X_t \cap \{ \ell=0 \}$ and a vector field $E_t'$ that deformation retracts it onto $P_t'$. 

\medskip
The vector field $E_t$ is obtained by gluing several vector fields on $X_t$ given by the lift by $\varphi_t: X_t \setminus \ell_t^{-1}(\{y_1(t), \dots, y_k(t) \}) \to D_t \setminus \{y_1(t), \dots, y_k(t) \}$ of suitable vector fields on the disk $D_t$. This is possible thanks to the first isotopy lemma of Thom-Mather, since the restriction of $\varphi_t$ to each stratum of the Whitney stratification of $X_t$ induced by the Whitney stratification of $X$ is of maximum rank. The resulting vector fields are rugose in the sense of \cite{Ve}, and hence they are integrable.

\medskip
Recall that the polyhedron $P_t$ is the union of the wings $C_j$ and the polyhedron $P_t'$ given by the induction hypothesis (see subsection \ref{subsection_CLP2}). Moreover, each wing $C_j$ consists of a collapsing cone $P_j$, a product $R_j$ and the gluing polyhedron $P_j'$ on $P_t'$, that is:
$$ C_j = P_j \cup R_j \cup P_j' \, . $$ 
See Figure \ref{Fig5}.

\medskip
Then it is natural that the construction of the vector field $E_t$ concerns at least three subsets of the Milnor fiber $X_t$: the points that are taken to $P_t' \backslash P_j'$ by the flow associated to $E_t$; the points that are taken to $P_j'$ and the points that are taken to $C_j \backslash P_j'$. This justifies the complexity of the construction given below.

\subsection{First step: decomposing $D_t$}
\ \\

Let $v_t$ be the continuous vector field on $D_t$, smooth and non-zero outside $Q_t$, transversal to $\partial D_t$ and zero on $Q_t$, defined in Lemma \ref{lemma_v_t}. If $q_t: [0,\infty[ \times \partial D_t \to D_t$ is the flow associated to $v_t$, then set:
$$V := D_t - q_t([0,A[ \times \partial D_t) \, ,$$
for some $A \gg 0$, which is a closed neighborhood of $Q_t$ whose boundary:
$$\partial V = q_t(\{A\} \times \partial D_t)$$ 
is smooth and transversal to each $\partial D_s(y_j(t))$ as in Figure \ref{fig8}.

\begin{figure}[!h] 
\centering 
\includegraphics[scale=0.5]{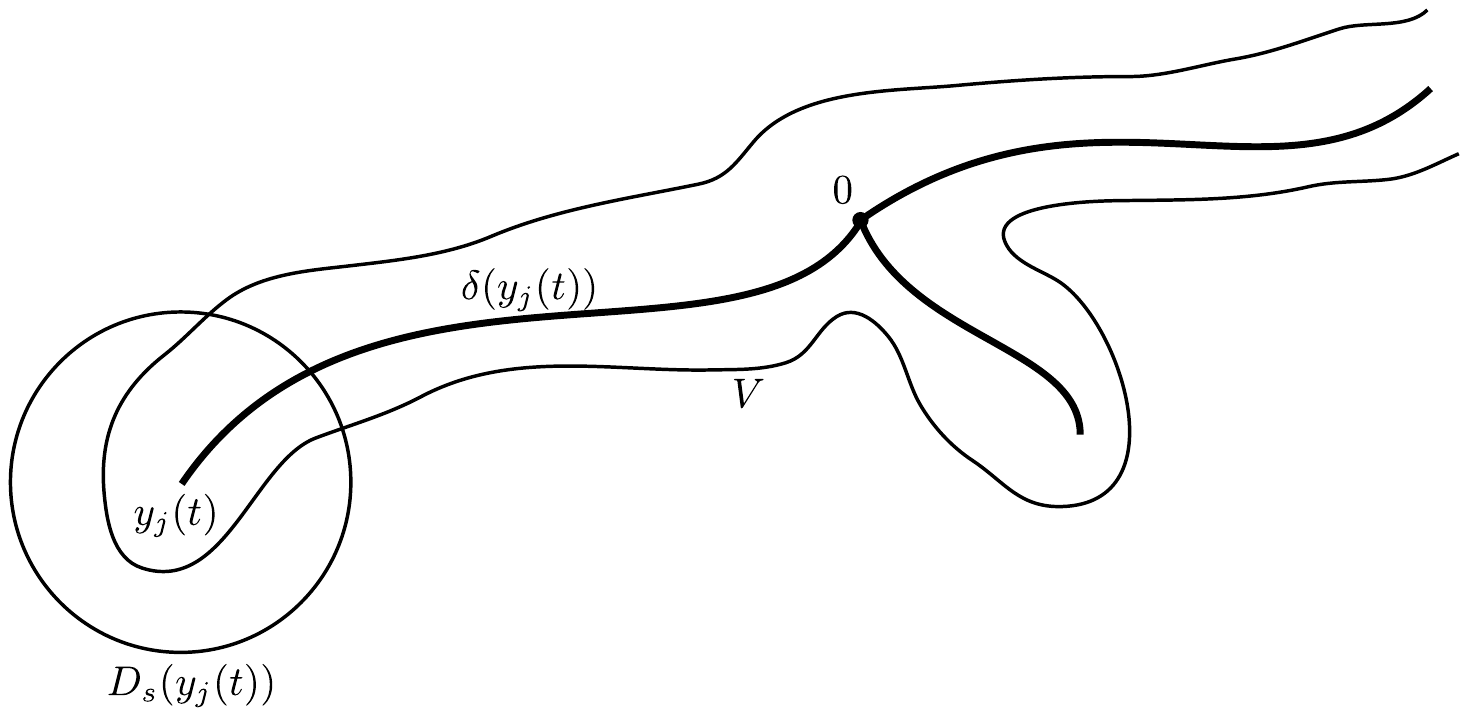}
\caption{}
\label{fig8}
\end{figure}

\medskip
Then we construct the vector field $E_t$ on $X_t$ by gluing a vector field $\tilde{\tilde{E_t}}$ in $\ell_t^{-1}(D_t \backslash V')$, 
where $V' := D_t - q_t([0,A'[ \times \partial D_t)$, with $A' > A$, $A'-A \ll 1$; and a vector field $\tilde{E_t}$ in $\ell_t^{-1}(V)$, using a partition of unity.

\medskip
The vector field $\tilde{\tilde{E_t}}$ in $\ell_t^{-1}(D_t \backslash V')$ is the lifting of the vector field $v_t$. It is integrable (see Lemma \ref{lemma_Ve}), and it is tangent to the strata of $\phi^{-1}(D_t \backslash V') \cap \mathring{\B}_\e$ and transversal (and pointing inwards) to the strata of $\partial X_t$. 

\medskip
The construction of the vector field $\tilde{E_t}$ in $\ell_t^{-1}(V)$ is much more complicated. We are going to do it in the rest of this subsection.

\medskip
\subsection{Second step: decomposing $V$}
\ \\

We first decompose $V$ into ``branches" $V_j$ as follows: each ``branch" $V_j$ is a closed neighborhood of $\delta(y_j(t)) \backslash \{0\}$ whose boundary is composed by $\partial V \cap V_j$ and two simple paths that one can suppose to be orbits of the vector field $v_t$ constructed above. See Figure \ref{fig9}. 

\begin{figure}[!h] 
\centering 
\includegraphics[scale=0.7]{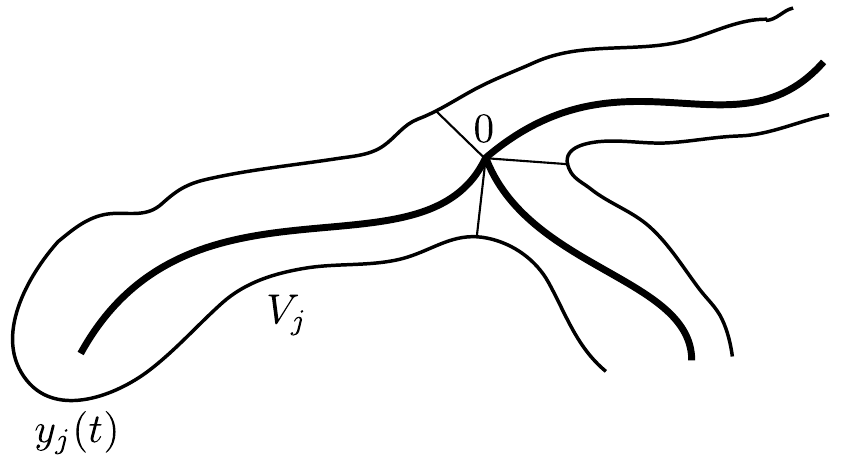}
\caption{}
\label{fig9}
\end{figure}

\medskip
We will construct the vector field $\tilde{E_t}$ by gluing the vector fields $\tilde{E}_{t,j}$ that we are going to construct on each $\ell_t^{-1}(V_j)$.  In other words, we will construct a vector field $\tilde{E}_{t,j}$ on $\ell_t^{-1}(V_j)$, for each $j$ fixed, which is continuous, integrable, tangent to the strata of $\mathcal{S}$, non-zero and smooth on $\ell_t^{-1}(V_j) \backslash C_j$, and zero on $C_j$, where $C_j$ is the polyhedron defined in subsection \ref{subsection_CLP2}.

\medskip
\subsection{Third step: covering $\ell_t^{-1}(V_j)$ by open sets $W_{j,i}$}
\ \\

Fix $j \in \{1, \dots, k\}$. The approach of the construction of each vector field $\tilde{E}_{t,j}$ will be the following: we will cover $\ell_t^{-1}(V_j)$ by open sets $W_{j,1}$, $W_{j,2}$, $W_{j,3}$ and $W_{j,4}$. Then we will construct the vector fields $\tilde{E}_{t,j,i}$ on $W_{j,i}$, for $i=1, \dots, 4$, in such a way that each orbit of the vector field $\tilde{E}_{t,j}$ obtained by gluing them with a partition of unity have a limit point at $P_t$.

\medskip
As before, given positive real numbers $r$ and $s$, let $B_r$ denote the ball around $x_j(t)$ in $\bC^N$ of radius $r$ and let $D_s$ denote the disk around $y_j(t)$ in $D_t$ of radius $s$.

\medskip
Let $r$ and $r'$ be small enough positive real numbers such that $r'<r$ and $r-r' \ll 1$. Let us cover $\ell_t^{-1}(V_j)$ by the open sets $W_{j,1}$, $W_{j,2}$, $W_{j,3}$ and $W_{j,4}$ defined as follows:
$$\bullet \ \ W_{j,1} := \ell_t^{-1}(\mathring{D}_s \cap \mathring{V_j}) \cap \mathring{B_r} \, $$
and
$$\bullet \ \ W_{j,2} := \ell_t^{-1}(\mathring{D}_s \cap \mathring{V_j}) \setminus B_{r'} \, .$$

\medskip
To define $W_{j,3}$ and $W_{j,4}$ we have to do a construction first. Set:
$$W_{j,3}':= \ell_t^{-1}(\mathring{V_j} \backslash D_{s'}) \, ,$$
where $s'<s$ and $s-s' \ll 1$. 

\medskip
We can construct a vector field $\omega_j$ in $V_j \backslash \mathring{D}_{s'}$ which is smooth, non zero outside $\{0\}$, zero on $\{0\}$, with trajectories transversal to $\partial V \cap (V_j \backslash D_{s'})$ and to $\partial D_{s'} \cap V_j$, as in Figure \ref{fig10}.

\begin{figure}[!h] 
\centering 
\includegraphics[scale=0.6]{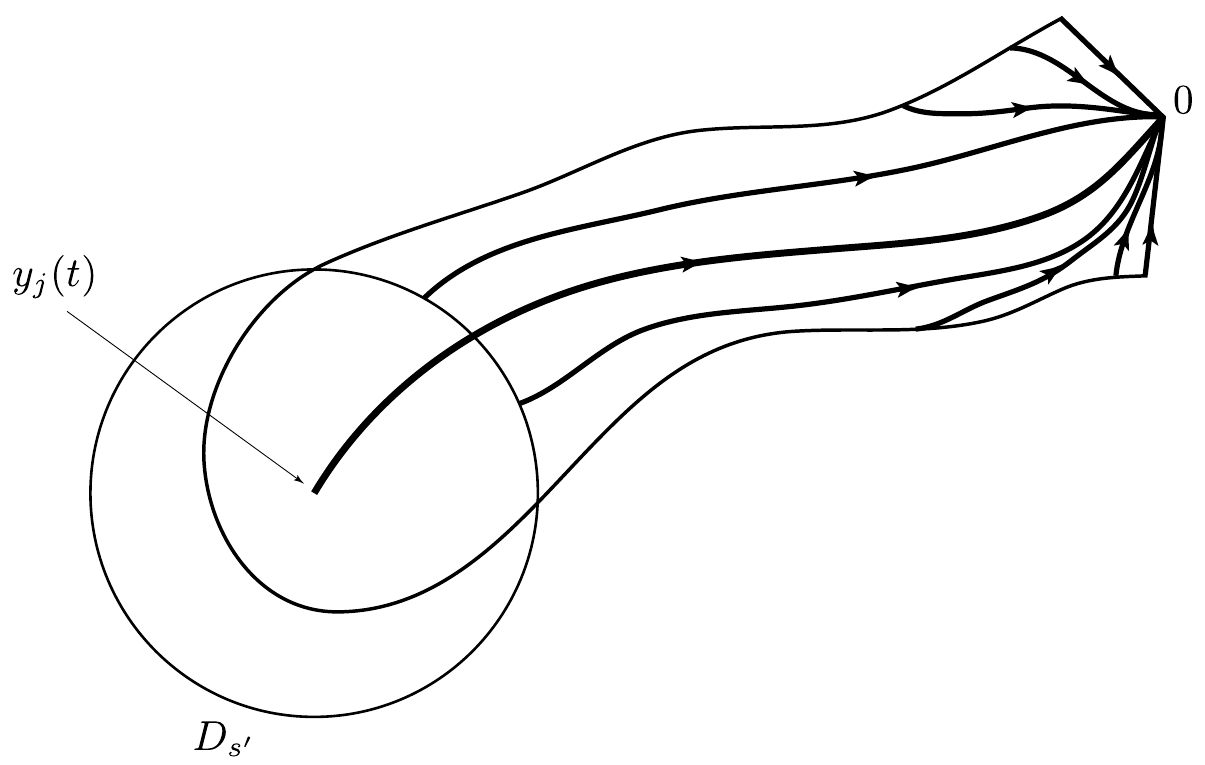}
\caption{}
\label{fig10}
\end{figure}

\medskip
Since the Whitney stratification $(S_\a)_{\a \in A}$ of $X$ induces a Whitney stratification on $\ell_t^{-1}(V_j \backslash \mathring{D}_{s'})$ and since the restriction of $\ell_t$ to $\ell_t^{-1}(V_j \backslash \mathring{D}_{s'}) \cap S_\a$ is a submersion for each $\a \in A$, we can lift $\omega_j$ to a continuous vector field $H_j$ in $\ell_t^{-1}(V_j \backslash \mathring{D}_{s'})$ that is rugose, smooth and tangent to each stratum, and that trivializes $\ell_t^{-1}(V_j \backslash \mathring{D}_{s'})$ over $V_j \backslash \mathring{D}_{s'}$. 

\medskip
Recall that the induction hypothesis applied to the restriction of $f$ to $X \cap \{\ell=0\}$ gives a vanishing polyhedron $P_t'$ in $X_t \cap \{ \ell=0 \}$ and a vector field $E_t'$ that deformation retracts $X_t \cap \{ \ell=0 \}$ onto $P_t'$.

\medskip
Then one can transport the vector field $E_t'$ on $\ell_t^{-1}(0)$ to all the fibers $\ell_t^{-1}(u)$ for $u \in V_j \backslash \mathring{D}_{s'}$. This way we obtain a vector field $\calV$ in $\ell_t^{-1}(V_j \backslash \mathring{D}_{s'})$ which is integrable, tangent to each stratum of $\ell_t^{-1}(u) \cap \mathring{\B}_\e$ and transversal to each stratum of $\ell_t^{-1}(u) \cap \BS_\e$, for any $u \in V_j \backslash \mathring{D}_{s'}$.

\medskip
Now consider the vector field $\vartheta_1$ in $\ell_t^{-1}(V_j \backslash \mathring{D}_{s'})$ given by:
$$\vartheta_1 := \calV + H_j \, ,$$
which is integrable, tangent to the strata of the stratification induced by $\mathcal{S}$, transversal to the strata of $\BS_\e \cap \ell_t^{-1}(V_j \backslash \mathring{D}_{s'})$, non-zero outside $P_t'$ and zero on $P_t'$.

\medskip
One can see (as in the case of the vector field $\vartheta_1$ of Subsection \ref{subsection_CLP2}) that the orbits of $\vartheta_1$ have a limit point in $P_t'$. The orbits $A(V_j,r)$ that intersect the set $\ell_t^{-1}(\mathring{V}_j \cap \partial D_s) \cap \mathring{B}_r$ by the action of $\vartheta_1$ is a set that we call $W_{j,4}'$. Then we define the set:
$$\bullet \ \ W_{j,4} := W_{j,4}' \cup W_{j,1} \, .$$

\medskip
Finally, the set $W_{j,3}$ is given by:
$$\bullet \ \ W_{j,3} := W_{j,3}' \backslash A(V_j',r') \, ,$$
where $r'<r$, with $r-r' \ll 1$, and $V_j' := D_t \backslash q_t([0, A'[ \times \partial D_t)$, with  $A'<A$ and $A-A' \ll 1$. One can check that both $W_{j,3}$ and $W_{j,4}$ are open sets.

\medskip
\subsection{Fourth step: constructing the vector fields $\tilde{E}_{t,j,i}$}
\ \\

\begin{itemize}
\item[$(1)$] \underline{Construction of $\tilde{E}_{t,j,1}$}:

\medskip
We can consider a smooth vector field $\tilde{\omega}$ on $V_j$ which is zero on $\delta(y_j(t))$, transversal to $\partial V_j$ and tangent to $\partial D_s \cap V_j$, like in Figure \ref{fig11}.

\begin{figure}[!h] 
\centering 
\includegraphics[scale=0.7]{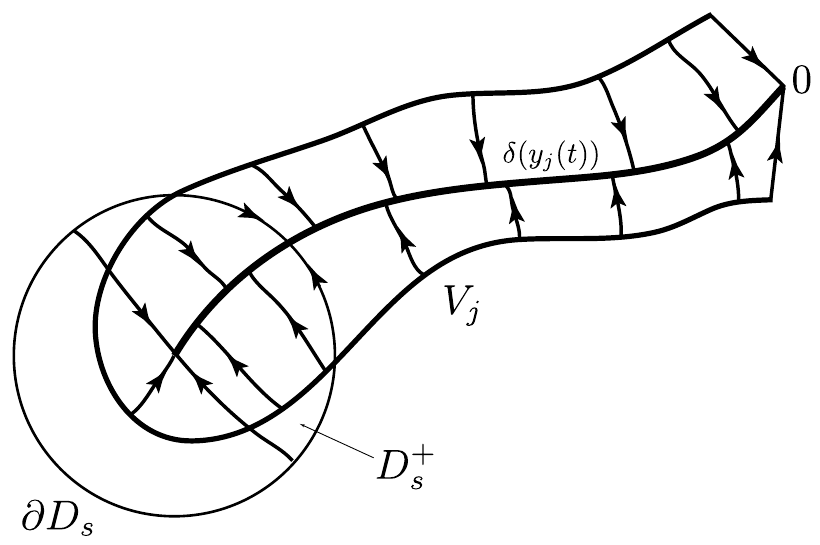}
\caption{}
\label{fig11}
\end{figure}

\medskip
Let $D_s^+$ be a semi-disk of $D_s$ which contains $\delta(y_j(t)) \cap \mathring{D}_s$ in its interior. 

\medskip
We will lift $\tilde{\omega}$ to a rugose vector field $\tilde{H}_j$ in $\ell_t^{-1}(\mathring{D}_s \cap \mathring{V}_j) \cap B_r$, which is zero on $\ell_t^{-1} \big( \mathring{D}_s \cap \delta(y_j(t)) \big)$, tangent to the strata of $\mathcal{S}$ and of $\ell_t^{-1}(\mathring{D}_s \cap \mathring{V}_j) \cap S_r$, where $S_r := \partial B_r$, in the following way: 

\medskip
Recall that we can apply the induction hypothesis to the restriction:
$$(\ell_t)_|: \ell_t^{-1}(D_s^+ \cap \mathring{D}_s \cap \mathring{V}_j) \cap \mathring{B}_r \to D_s^+ \cap \mathring{D}_s \cap \mathring{V}_j  \, ,$$
which has an isolated singularity at $x_j(t)$ in the stratified sense, since \\ $\ell_t^{-1}(D_s^+ \cap \mathring{D}_s \cap \mathring{V}_j)$ has complex dimension $n-1$, where $n$ is the dimension of $X$. 

\medskip
Then we obtain a vector field $E_j$ and a collapsing cone $P_j^+$. Let:
$$ \psi: [0,+\infty[ \times \big( \ell_t^{-1}(\mathring{D_s^+} \cap \mathring{V}_j) \cap S_r \big) \to \ell_t^{-1}(\mathring{D_s^+} \cap \mathring{V}_j) \cap B_r$$
be the flow associated to $E_j$ and set:
$$P_j(u) := \psi \big( \{u\} \times \ell_t^{-1}(\mathring{D_s^+} \cap \mathring{V}_j) \cap S_r \big) \, ,$$
where $u \geq 0$. 

\medskip
The Whitney stratification $\mathcal{S}$ induces a Whitney stratification of $P_j(u)$ (see Lemma \ref{lemma_strata} and notice that $P_j(0) = \ell_t^{-1}(\mathring{D_s^+} \cap \mathring{V}_j) \cap S_r$ is the intersection of $\ell^{-1}(\mathring{D_s^+} \cap \mathring{V}_j)$ with the stratified space $X_t \cap S_r$). Moreover, the restriction of $\ell_t$ to each stratum has maximum rank. So by Lemma \ref{lemma_Ve} we can lift the vector field $\tilde{\omega}$ over $\mathring{D_s^+} \cap \mathring{V}_j$ to a vector field that is tangent to the strata of $P_j(u)$. 

\medskip
On the other hand, for any point in $\ell_t^{-1}\big( (\mathring{D}_s \backslash \mathring{D_s^+}) \cap \mathring{V}_j \big)$ we just ask the vector field $\tilde{H}_j$ to be tangent to the strata of $\mathcal{S}$ and to lift $\tilde{\omega}$. This can be done locally and then $\tilde{H}_j$ is obtained by a partition of unity.

\medskip
Notice that at any point of $\mathring{B}_r \cap \ell_t^{-1} \big( \mathring{D}_s^+  \cap \mathring{V}_j \big) - \{x_j(t)\}$ and at any point of $(\mathring{B}_r \backslash B_{r'}) \cap \ell_t^{-1} \big( (\mathring{D}_s \backslash D_s^+) \cap \mathring{V}_j \big)$, for $r'<r$ with $r-r'\ll 1$, one can extend the vector field $E_j$ on a small open neighborhood.

\medskip
Now we construct $\tilde{E}_{t,j,1}$ as follows:

\medskip
\begin{itemize}
\item[$\circ$] Over a small open neighborhood $U_{x_j(t)}$ of $x_j(t)$, consider the zero vector field.

\item[$\circ$] For any $z \in \mathring{B}_r \cap \ell_t^{-1}(D_s^+ \cap \mathring{D}_s \cap \mathring{V}_j) \backslash \{x_j(t)\}$, take an open neighborhood $U_z$ of $z$ small enough such that it does not contain $x_j(t)$, it is contained in $\mathring{B}_r \cap \ell_t^{-1}(\mathring{D}_s \cap \mathring{V}_j)$ and $E_j$ is well defined on it. Then in $U_z$ we define the vector field:
$$E_{1,z} := E_{j |U_z} + \tilde{H}_{j_{|U_z}} \, .$$
This vector field is rugose, tangent to the strata of $\mathcal{S}$, non-zero outside the intersection of $U_z$ and $P_j$ and zero on $P_j \cap U_z$, where: 
$$P_j := P_j^+ \cap \ell_t^{-1}\big( \delta(y_j(t)) \big) \, .$$

\item[$\circ$] For any $z \in \mathring{B}_{r'} \cap \ell_t^{-1} \big( (\mathring{D}_s \backslash D_s^+) \cap V_j \big)$, take a small open neighborhood $U_z$ of $z$ and set
$$E_{1,z} := \tilde{H}_{j_{|U_z}} .$$

\item[$\circ$] For any $z \in (\mathring{B}_r \backslash B_{r'}) \cap \ell_t^{-1} \big( (\mathring{D}_s \backslash D_s^+) \cap V_j \big)$, take a small open neighborhood $U_z$ of $z$ contained in $(\mathring{B}_r \backslash B_{r'}) \cap \ell_t^{-1} \big( (\mathring{D}_s \backslash D_s^+) \cap V_j \big)$ and set:
$$E_{1,z} := E_{j |U_z} + \tilde{H}_{j_{|U_z}} \, .$$

\item[$\circ$] Then considering a partition of unity $(\theta_z)$ associated to the covering $(U_z)$, we set the vector field:
$$\tilde{E}_{t,j,1} := \sum \theta_z E_{1,z}$$
in $\ell_t^{-1}(\mathring{D}_s \cap \mathring{V}_j) \cap \mathring{B}_r$, which is continuous, rugose outside the point $x_j(t)$ (and therefore in $W_{j,1} \backslash P_j$), tangent to the strata of $\mathcal{S}$, non-zero outside $P_j$ and zero on $P_j \cap \big( \ell_t^{-1}(\mathring{D}_s \cap \mathring{V}_j) \big) \cap \mathring{B}_r$.

\item[$\circ$] Notice that if $z \in \ell_t^{-1}(\mathring{D}_s \backslash D_s^+) \cap \mathring{B}_r$, its orbit by $\tilde{E}_{t,j,1}$ has $\{x_j(t)\}$ as limit point, and the orbit by $\tilde{E}_{t,j,1}$ of a point $z \in \ell_t^{-1}(D_s^+) \cap \mathring{B}_r$ has its limit point in $P_j$.
\end{itemize}

\medskip
\item[$(2)$] \underline{Construction of $\tilde{E}_{t,j,2}$}:

\medskip
Consider a smooth non-zero vector field $\hat{\omega}_j$ in $\mathring{V}_j \cap \mathring{D}_s$ as Figure \ref{fig12} and such that, for any $u \in \mathring{V}_j \cap \mathring{D}_s$ one has the following implication:
$$\lambda \tilde{\omega}(u) + \mu \hat{\omega}_j(u) =0, \ \ \lambda \geq0, \ \ \mu \geq0 \implies \lambda=\mu=0 \, ,$$
where $\tilde{\omega}$ is the vector field defined above.

\begin{figure}[!h] 
\centering 
\includegraphics[scale=0.7]{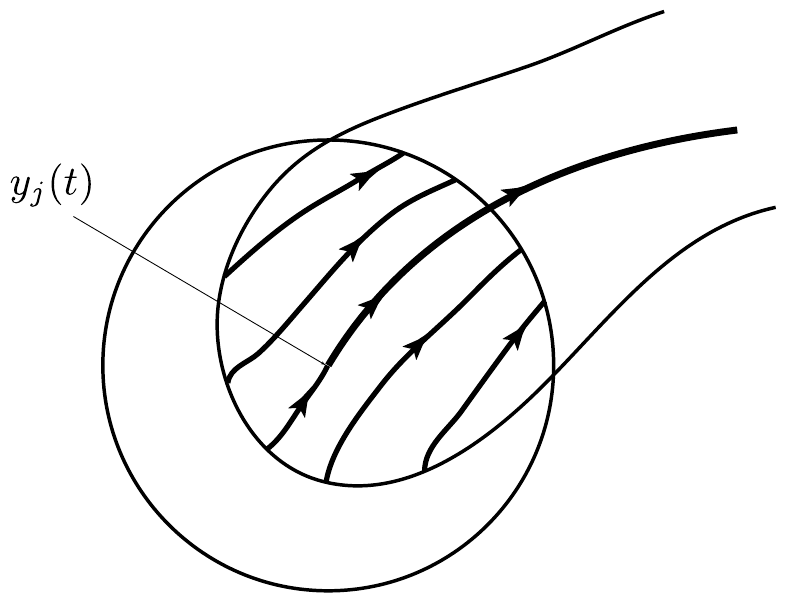}
\caption{}
\label{fig12}
\end{figure}

\medskip
Then $\tilde{E}_{t,j,2}$ is the lift of $\hat{\omega}_j$ in $W_{j,2}$, which is rugose and tangent to the strata of $\mathcal{S}$ and of $\ell_t^{-1}(\mathring{D}_s) \cap S_r$.

\medskip
\item[$(3)$] \underline{Construction of $\tilde{E}_{t,j,3}$}:

\medskip
We set $\tilde{E}_{t,j,3}$ to be the restriction of the vector field $\vartheta_1$ constructed above to $W_{j,3}$.

\medskip
\item[$(4)$] \underline{Construction of $\tilde{E}_{t,j,4}$}:

\medskip
Over $W_{j,4}'$, consider the continuous vector field $E_4'$ obtained by the restriction of $E$ to $\ell_t^{-1}(D_s^+) \cap \mathring{B}_r$ transported by the action of $\vartheta_1$. It is rugose and tangent to the strata of $\mathcal{S}$.

\medskip
Over $W_{j,4} = W_{j,4}' \cup W_{j,1}$, the vector field $E_4'$ glues with $\tilde{I}_{j,1}$, resulting a vector field $\tilde{E}_{t,j,4}$, which is continuous, rugose and non-zero on $W_{j,4} \backslash P_t$. The orbits of the points of $W_{j,4}$ by $\tilde{E}_{t,j,4}$ has limit points in $P_t$.

\end{itemize}

\medskip
\subsection{Fifth step: gluing all the vector fields to obtain $E_t$}
\ \\

Now, considering $\psi_2$, $\psi_3$ and $\psi_4$ a partition of unity associated to $W_{j,2}$, $W_{j,3}$ and $W_{j,4}$, we obtain the vector field:
$$\tilde{E}_{t,j} := \psi_2 \tilde{E}_{t,j,2} + \psi_3 \tilde{E}_{t,j,3} + \psi_4 \tilde{E}_{t,j,4}$$
in $\ell_t^{-1}(\mathring{V}_j)$, which is continuous, rugose, non-zero on $\ell_t^{-1}(V_j) \setminus P_t$ and zero on $P_t$.

\medskip
Gluing these vector fields $\tilde{E}_{t,j}$, for $j=1, \dots, k$, we get the vector field $\tilde{E}_{t}$. 

\medskip
Finally, gluing the vector field $\tilde{E}_{t}$ in $\ell_t^{-1}(V)$ and the vector field $\tilde{\tilde{E}}_t$ in $\ell_t^{-1}(D_t \backslash V')$ constructed in Subsection 5.1, we obtain a continuous vector field $E_t$ in $X_t$ with the properties $(i)$ to $(v)$ of Proposition \ref{prop_Et}. We just have to check that the orbits of this vector field have a limit point when the parameter goes to the infinity:

\medskip
\begin{itemize}
\item[$(a)$] If $z \in \ell_t^{-1}(D_t \backslash V')$, the orbit of $z$ arrives to $W_{j,2} \cup W_{j,3} \cup W_{j,4}$ after a finite time.
\item[$(b)$] If $z \in W_{j,2}$, the orbit of $z$ arrives to $W_{j,3} \cup W_{j,4}$ after a finite time.
\item[$(c)$] If $z \in W_{j,3} \backslash W_{j,4}$, it has a limit point on $\cup_{j=1}^k C_j$.
\item[$(d)$] If $z \in W_{j,4} \backslash W_{j,3}$, it has a limit point on $P_t'$.
\item[$(e)$] If $z \in W_{j,3} \cap W_{j,4}$, we have that the orbit passing through $z$ has a limit point that is the limit point by $\vartheta_1$ of the limit point of the orbit of $z$ by $E_4'$. Hence this limit point is on $P_j' = P_t' \cap \overline{C}_j$.
\end{itemize}



\vspace{0.5cm}
\begin{thebibliography}{99}

\bibitem{Be1} E. Bertini, {\it Introduction to the projective geometry of hyperspaces}, Messina (1923).

\bibitem{Be2} E. Bertini, {\it E. Bertini, Algebraic surfaces}, Trudy Mat. Inst. Steklov, 75 (1965).

\bibitem{Ca} C. Carath\'eodory, {\it \"Uber die gegenseitige Beziehung der R\"ander bei der konformen Abbildung des Inneren einer Jordanschen Kurve auf einen Kreis}, Mathematische Annalen (Springer Berlin / Heidelberg) 73 (2): 305-320 (1913).

\bibitem{Fo} R.L. Foote, {\it Regularity of the distance function}, Proc. of the A.M.S., Vol. 92, Number 1 (1984).

\bibitem{Gr} G. Grauert, {\it Ein Theorem der analytischen Garbentheorie und die Modulraume Komplexer Strukturen} Publ. Math. IHES, 5 (1960).

\bibitem{Ho} C. Houzel, {\it Geom\'etrie analytique locale}, S\'eminaire H. Cartan, 1960/61, expos\'es 18-21.

\bibitem{Le1} D.T. L\^e, {\it Calcul du nombre de cycles \`evanouissants d'une hypersurface complexe}, Ann. Inst. Fourier 23 (1973), 261-270, Grenoble, France.

\bibitem{Le2} D.T. L\^e, {\it Vanishing cycles on analytic sets}, Proc. Conf. on Algebraic Analysis, RIMS, Kyoto, July 1975.

\bibitem{Le3} D.T. L\^e, {\it Remarks on relative monodromy}, in ``Real and complex singularities", ed. by P. Holm, Sitjhoff and Nordhoff 1977.

\bibitem{Le4} D.T. L\^e, {\it Poly\`edres \`evanescents et effondrements}, A f\^ete of topology, 293-329, Academic Press, Boston, MA, 1988.

\bibitem{Le5} D.T. L\^e, {\it Complex Analytic Functions with Isolated Singularities}, J. Algebraic Geometry 1 (1992), 83-100.

\bibitem{LT} D.T. L\^e and B. Teissier, {\it Cycles evanescents, sections planes et conditions de Whitney. II.}, Singularities, Part 2 (Arcata, Calif., 1981), 65-103, Proc. Sympos. Pure Math., 40, Amer. Math. Soc.

\bibitem{Ma} J. Mather, {\it Notes on topological stability}, Harvard, 1970 (unpublished).

\bibitem{Mi} J.W. Milnor, {\it Singular points of complex hypersurfaces}, Ann. of Math. Studies 61, Princeton, 1968.

\bibitem{Mu} J.R. Munkres, {\it Topology}, Pearson; 2 edition (January 7, 2000).

\bibitem{NP} R. Nevanlinna and V. Paatero, {\it Introduction to complex analysis}, second ed., New York: Chelsea, cop. (1982).

\bibitem{Te} B. Teissier, {\it Cycles \'evanescents, sections planes et conditions de Whitney}, Ast\'erisque 7/8 (1973) Soc. Math. Fr.

\bibitem{Te2} B. Teissier, {\it Multiplicit\'es polaires, sections planes et conditions de Whitney}, Lecture Notes in Math. 961 (1982), Springer-Verlag, Berlin, R.F.A.

\bibitem{Th} R. Thom, {\it Ensembles et morphismes stratifi\'es}, Bull. Amer. Math. Soc. 75 (1969), 240-284. 

\bibitem{Ve} J.L. Verdier, {\it Stratifications de Whitney et th\'eor\`eme de Bertini-Sard}, Inv. Math. 36 (1976), 295-312.

\bibitem{Wh} H. Whitney, {\it Tangents to an analytic variety}, Ann. of Math. (2) 81 (1965) 496-549.

\end{thebibliography}
\end{document}